\documentclass{amsart}%
\usepackage{amsfonts}%
\usepackage{amsmath}%

\usepackage{amscd}
\usepackage{amsthm}
\usepackage{epsfig}
\usepackage{amstext}
\usepackage[all]{xy}
\usepackage{graphicx}%
\usepackage{amssymb}%
\usepackage{graphicx}

\usepackage{makecell}
\usepackage{xcolor}

\newtheorem{theorem}{Theorem}[section]
\theoremstyle{plain}

\newtheorem{corollary}[theorem]{Corollary}

\newtheorem{definition}[theorem]{Definition}
\newtheorem{example}[theorem]{Example}

\newtheorem{lemma}[theorem]{Lemma}

\newtheorem{proposition}[theorem]{Proposition}
\newtheorem{remark}[theorem]{Remark}

\newtheorem{question}[theorem]{Question}
\numberwithin{equation}{section}

\newcommand{\C}{\mathbb{C}}
\newcommand{\D}{\mathbb{D}}
\newcommand{\G}{\mathcal{G}}

\newcommand{\R}{\mathbb{R}}
\newcommand{\Q}{\mathbb{Q}}

\newcommand{\Z}{\mathbb{Z}}

\newcommand{\ch}[0]{\operatorname{ch}}
\newcommand{\ev}[0]{\operatorname{ev}}
\newcommand{\id}[0]{\operatorname{id}}
\renewcommand{\ker}[0]{\operatorname{ker}}

\newcommand{\bott}[0]{\operatorname{bott}}
\newcommand{\Hom}[0]{\operatorname{Hom}}

\begin{document}


\title{Almost commuting matrices, cohomology, and dimension}

\date{\today}

\author{Dominic Enders}
\author{Tatiana Shulman}

\address{Dominic Enders
\newline Mathematisches Institut der WWU M{\"u}nster, Deutschland.}
\email{d.enders@uni-muenster.de}

\address{Tatiana Shulman
\newline University of Gothenburg, Sweden.}
\email{tatshu@chalmers.se}




\begin{abstract}
 It is an old problem to investigate which relations for families of commuting matrices are stable under small perturbations, or in other words, which commutative $C^*$-algebras $C(X)$ are matricially semiprojective.
Extending the works of Davidson, Eilers-Loring-Pedersen, Lin and Voiculescu on almost commuting matrices, we identify the precise dimensional and cohomological restrictions for finite-dimensional spaces $X$ and thus obtain a complete characterization:
$C(X)$ is matricially semiprojective if and only if $\dim(X)\leq 2$ and $H^2(X;\mathbb{Q})=0$.\\
We give several applications to lifting problems for commutative $C^*$-algebras, in particular to liftings from the Calkin algebra and to $\ell$-closed $C^*$-algebras in the sense of Blackadar.

\medskip

In French:

\noindent Matrices commutant presque, cohomologie et dimension.

\noindent R\'esum\'e. Un vieux probl\'eme consiste \'a chercher, pour des familles de matrices commutant quelles relations sont stables sous de petites perturbations, ou en d'autres termes, quelles C*-alg\'ebres commutatives $C(X)$ sont matriciellement semi-projectives. En prolongeant les travaux de Davidson, Eilers-Loring-Pedersen, Lin et Voiculescu sur les matrices commutant presque, nous identifions les restrictions dimensionnelles et cohomologiques pr\'ecises pour l'espace de dimension finie $X$ et obtenons ainsi un caract\'erisation compl\'ete: $C(X)$ est matriciellement semi-projectif si et seulement si $\dim(X) \leq 2$ et $H^2(X;\mathbb{Q})=0$.\\
 Nous donnons plusieurs applications aux probl\'emes de rel\'evement pour les C*-alg\'ebres commutatives, en particulier aux rel\'evement de l'alg\'ebre de Calkin et aux C*-alg\'ebres $\ell$-ferm\'es dans le sens de Blackadar.

\end{abstract}

\maketitle


\section{Introduction}
Questions concerning almost commuting matrices have been studied for many decades.
Originally, these types of questions first appeared in the 1960s and were popularized by Paul Halmos who included one such question in his famous list of open problems \cite{Halmos}. Specifically he asked:

\medskip

{\it Is it true that for every $\varepsilon > 0$ there is a $\delta > 0$ such that if matrices $A, B$ satisfy $$\|A\|\le 1, \|B\|\le 1 \;\text{and}\; \|AB-BA\|< \delta,$$ then the distance from the pair $A, B$ to the set of commuting pairs is less than $\varepsilon$?
(Here $\varepsilon$ should not depend on the size of matrices, that is on the dimension of the underlying space. The norm above is the operator norm.)}

\medskip

{\it Is the same true under the additional assumption that $A$ and $B$ are Hermitian?}

\medskip

The first significant progress was made by Voiculescu \cite{Voiculescu} who observed that almost commuting unitary matrices need not be close to  commuting unitary matrices.
Later, in \cite{Choi}, Choi gave an example of almost commuting matrices which were not close to any commuting ones, thus answering the first of Halmos's questions above.
Shortly thereafter Exel and Loring found an alternative, short and elementary proof of Voiculescu's result which in addition revealed that Voiculescu's pair was also not close to any pair of commuting matrices \cite{ExelLoring}.

Already in \cite{Voiculescu}, Voiculescu explained that questions about almost commuting matrices have a C*-algebraic nature and relate to a particular type of lifting property.
To make this precise,  consider the C*-algebra
$$\prod M_n(\C) = \{(T_n)_{n\in \mathbb N}\;|\; T_n \in M_n(\mathbb C), \;\sup_n \|T_n\|<\infty\}$$
together with the ideal $\bigoplus M_n(\C)$ of sequences $(T_n)_{n\in \mathbb N}$ with $\lim_{n\to \infty} \|T_n\| = 0$.
Then the question of whether matrices almost satisfying some relations $\mathcal{R}$ are always close to matrices exactly satisfying these relations can be reformulated as a lifting problem for the corresponding universal C*-algebra  $C^*(\mathcal{R})$ of these relations:
$$
\xymatrix{& \prod M_{n}(\C) \ar@{->>}[d]\\
 C^*(\mathcal{R}) \ar@{-->}[ur] \ar[r] & \prod M_{n}(\C)/\bigoplus M_{n}(\C)}
 $$

\noindent Voiculescu discusses the question of whether $C(X)$ has this property for various compact metrizable spaces $X$.
However, the lifting property as described above makes perfect sense for arbitrary, not necessarily commutative C*-algebras as well and has been given many different names by various authors -- matricial stability, matricial weak semiprojectivity, matricial semiprojectivity, just to name a few. We adopt the latter name here.

In this terminology, Voiculescu's result on two almost commuting unitaries states that $C(\mathbb T^2)$, the universal C*-algebra for two commuting unitaries, is actually not matricially semiprojective.
Similarly, Halmos's question about almost commuting Hermitian matrices now reads as whether $C(X)$ is matricially semiprojective for $X=[0,1]^2$. This problem remained unanswered for a very long time, but was eventually solved by Lin in the positive \cite{Lin}.
Many further results and techniques around matricial semiprojectivity of $C(X)$ have been established over time, and numerous applications to C*-theory have been found.
It is impossible to trace and name them all here, so we only mention \cite{Winter} and \cite{Willett} as two of the most recent examples to illustrate the point.

Beyond the world of C*-algebras, almost commuting matrices have found their way into many other areas of mathematics. Striking applications can be found in operator theory (see for instance  \cite{FriisRordamBDF, BergDavidson, KachkovskySafarov}), quantum physics and condensed matter physics (see e.g. \cite{HastingsLoring} and references therein), and even computer science (see e.g. \cite{Computing2} and references therein).
Variations in which matrices almost commute with respect to various norms other than the operator norm have also found applications, e.g. in group theory  (see for instance \cite{ESS18, HS18, CGLT, BL} and references therein).

\vfill
\pagebreak

Coming back to the central question of this paper -- {\it For which compact metrizable space $X$ is C(X) matricially semiprojective?} -- we summarize the state of the art, i.e. the main examples for which the answer is known, below.

 \begin{center}
 \begin{tabular}{||c| c| c ||}
 \hline
 \thead{Compact space $X$} & \thead{Is $C(X)$ matrici- \\ ally semiprojective?} & \thead{Reference}  \\ [0.5ex]
 \hline\hline
 $\mathbb T^2$&No&Voiculescu  \cite{Voiculescu}, a short proof\\
  (2-torus)&  &  by Exel and Loring  \cite{ExelLoring}\\
 \hline
 $[0,1]^2$ & Yes  & Lin  \cite{Lin}, a short proof\\
  &  & by Friis and R{\o}rdam  \cite{FriisRordam}\\
 \hline
  $[0,1]^3$ & No &Voiculescu  \cite{Voiculescu81}, \\
 &&Davidson  \cite{Davidson3Herm}\\
 \hline
$S^2$ & No &Voiculescu (\cite{Voiculescu81} + a remark  \\
(2-dimensional sphere)&& in \cite{Voiculescu}), Loring  \cite{Lor88}\\
\hline
$\mathbb RP^2$  & Yes & Eilers, Loring, Pedersen \\
(real projective plane) &&\cite{ELP}\\
 \hline
  1-dimensional  & Yes & Loring   \cite{Loring89}\\
  CW-complexes & &  \\
  \hline
 \end{tabular}
\end{center}

\bigskip

Despite receiving great attention, this problem still remains open.
In this paper we solve the problem under the very mild additional assumption that $X$ has finite covering dimension. (In terms of generators and relations, this means we only look at {\it finite} families of matrices (almost) satisfying possibly infinitely many relations (see Prop. \ref{CovDimFinGen}).)\\

\medskip

Before stating our main result, we would like to mention a few other important lifting properties for C*-algebras which are related to matricial semiprojecivity  -- projectivity, semiprojectivity, and weak semiprojectivity (see section \ref{Preliminaries} for definitions).
In contrast to matricial semiprojectivity, the question when $C(X)$ has one of these other lifting properties has been successfully resolved \cite{ChDr, AdamHannes, Dominic}.
An important first step in the solutions to all these cases was to obtain information on $X$ by restricting to the category of commutative C*-algebras, i.e. by only considering lifting problems with commutative targets.
This immediately led to the understanding that $C(X)$ can be projective only when $X$ is an absolute retract, and that $C(X)$ can be (weakly) semiprojective only when $X$ is an  (approximate) absolute neighborhood retract. Thus in all three cases the spaces involved were reasonably well-behaved, and this information was fully used later on.
We want to point out that nothing like this is possible for matricial semiprojectivity. Restricting to the commutative subcategory gives no extra information in this case (we make this precise in Remark \ref{MWSPCommCategory}) and we are thus forced to cope with general spaces, possibly without any form of regularity.
This is one of the main points which distinguishes matricial semiprojectivity from these other lifting properties and which makes it highly intractable.

In this paper, it is therefore crucial to develop a technique which allows us to keep track of our lifting property of interest while approximating a space by nicer, more regular ones. Using approximations of compact, metrizable spaces by CW-complexes, we manage to reduce the problem to this more tractable case.
Even for this case the question is still open, but there is a very useful result by Eilers, Loring and Pedersen \cite{ELP} which
 gives sufficient conditions for matricial semiprojectivity of 2-dimensional (noncommutative) CW-complexes in terms of K-theory.  Furthermore, in \cite{GongLin} Gong and Lin provide sufficient conditions for matricial semiprojectivity of compact metrizable spaces (not necessarily CW-complexes!) of dimension not larger than 2 in terms of KK-theory.  Our main result stated below in particular will show that the sufficient conditions of Eilers, Loring and Pedersen for matricial semiprojectivity of commutative 2-dimensional CW-complexes are also necessary. At this moment  it is unclear to us whether these agree with the conditions of Gong and Lin in \cite{GongLin}.

So what answer can one expect? The dimension of the underlying space was expected play a role after the negative answer for the 3-dimensional cube \cite{Voiculescu81, Davidson3Herm}.
Besides that, it was already hinted by Voiculescu in \cite{Voiculescu} that vanishing second cohomologies of the underlying space $X$ should be of importance here (see also \cite{Lor88}). This is not exactly right as the example of $\mathbb RP^2$ shows, but it is quite close  -- coefficients will also play a role. For spaces of finite covering dimension we give an answer here:\\

\noindent {\bf Main theorem.} Let $X$ be a compact metrizable space with $\dim X < \infty$. Then $C(X)$ is matricially semiprojective if and only if $\dim X \leq 2$ and $H^2(X; \Q) = 0$.\\

It is not clear whether the finiteness assumption on the dimension of $X$ is actually necessary. It is only used to guarantee the existence of  closed subsets of any smaller dimension. In particular, the main theorem will also hold for any infinite-dimensional space that contains a closed subspace of dimension 3 or larger.
Once the existence of closed subsets of smaller dimension is guaranteed, we can further establish the necessity of the dimensional restrictions stated in the main theorem.
To this end, we prove an auxiliary, purely topological result. This statement might be of independent interest, so we state it here.\\

\noindent {\bf Theorem 3.7.}  Suppose $n <  \dim X < \infty$. Then there exists a closed subset $A$ of $X$ such that $\dim A=n$ and $H^n(A, \mathbb Q)\neq 0$.\\

Besides topological ingredients, our proof is obtained by operator-algebraic methods. In particular the necessity part draws on some ideas from classification of C*-algebras (embeddings into AF-algebras with controlled K-theory (\cite{DL92})).

Despite being an active research topic for half a century, matricial semiprojectivity of general C*-algebras still lacks a well-developed theory, due to technical difficulties arising in the study of its permanence properties.
The main result of this paper can be considered as a step towards a systematic investigation of this concept.

\medskip

Let us mention some further applications of our main result and of the techniques developed along the way.
First, we obtain new results concerning liftings from the Calkin algebra $Q(H)$.
Much of our understanding of operators on Hilbert spaces comes from analyzing their images in the Calkin algebra instead.
Vice versa, often one studies the Calkin algebra by ``lifting its properties'', that is, by showing that an element has a certain property only when some operator in the corresponding coset does.
The list of works on this topic is immense, but the celebrated work of Brown-Douglas-Fillmore \cite{BDF} needs of course to be mentioned here.
BDF-theory originated as a classification of essentially normal operators, i.e.\ normal elements in the Calkin algebra, by invariants somewhat related with Fredholm index, and was the first instance of algebraic topology methods being applied to C*-algebras.
It is concerned with the study of extensions, i.e.\ with  injective $\ast$-homomorphisms from commutative C*-algebras into the Calkin algebra $Q(H)$, and in particular studies the question when such extensions are trivial, that is, when they can be lifted to $B(H)$. The BDF-machinery requires the $\ast$-homomorphism to be injective, so that liftability questions for possibly non-injective maps are not covered and corresponding results appear to be unknown.
We use our main theorem to give, for 1-dimensional spaces, a necessary and sufficient condition for liftability in this more general case.

\medskip

\noindent {\bf Theorem 5.9.} Let $X$ be a compact metrizable space and $\dim X \le 1$. The following are equivalent:

(1) All $\ast$-homomorphisms from $C(X)$ to $Q(H)$ are liftable;

(2) $\Hom(H^1(X, \mathbb Z), \mathbb Z) = 0$.

\medskip

When applied to planar sets, Brown-Douglas-Fillmore theory gives a classification of essentially normal operators and in particular answers the question of when a normal element of the Calkin algebra lifts to a normal operator on $B(H)$.
However, conditions under which all normal elements with spectrum contained in a given subset of the plane are liftable seem to be unknown.
Below we also give necessary and sufficient conditions for that.

\medskip

\noindent {\bf Proposition 5.11.} Let $X$ be a compact subset of the plane. The following are equivalent:

(i) Any normal element of the Calkin algebra with spectrum contained in $X$ lifts to a normal operator;

(ii) Any normal element of the Calkin algebra with spectrum contained in $X$ lifts to a normal operator with the same spectrum;

(iii) $\dim X\le 1$ and $H^1(X, \mathbb Z)=0$.

\medskip

One further application is concerned with $\ell$-closed C*-algebras as introduced by Blackadar in \cite{Blackadar}.
He defined $\ell$-closed and $\ell$-open C*-algebras as noncommutative analogues of topological spaces with certain natural properties regarding extendability of continuous maps (see Section 5.2 for definitions).
It turns out that $\ell$-openness is, as was shown by Blackadar, closely related to semiprojectivity.
$\ell$-closedness, on the other hand, remains quite mysterious.
There is no idea yet how to characterize it, but Blackadar states that {\it``It seems reasonable that if $X$ is any ANR (absolute neighborhood retract), then $C(X)$ is $\ell$-closed''}.
We show that this is actually not the case, namely that dimensional obstructions need to be taken into account.

\medskip

\noindent {\bf Corollary 5.15.} Let $X$ be a CW-complex. If $C(X)$ is $\ell$-closed then $\dim X \le 3$.

\medskip

\bigskip

This paper is organized as follows. Section 2 contains preliminaries on matricial semiprojectivity. In Section 3 we prove our auxiliary topological result, while Section 4 contains the proof of the main theorem. Applications are discussed in Section 5.\\

\textbf{Acknowledgements.}
We are grateful to S{\o}ren Eilers for a very useful discussion on the role of torsion for stability of NC CW-complexes. We are grateful to the referee and to the editor of this paper for their suggestions and careful revision, that led to a better exposition of the paper.

The work of the first-named author was supported by the SFB 878 {\it Groups, Geometry and Actions}.
The work of the second-named author was supported by the Swedish Research Council, by the Polish National Science Centre grant under the contract number 2019/34/E/ST1/00178 and by the grant H2020-MSCA-RISE-2015-691246-Quantum Dynamics.
Finally the work of both authors was supported by the Danish National Research Foundation through the {\it Centre for Symmetry and Deformation} (DNRF92).


\section{Preliminaries}\label{Preliminaries}

\subsection{Matricial semiprojectivity and related lifting properties of C*-algebras}\label{sec:mwsp}
By an ideal of a C*-algebra we always will mean a closed two-sided ideal.
For  a C*-algebra $A$ and an ideal $I$ in $A$,  the quotient map $A\to A/I$  we be denoted by $\pi$.

The following four lifting properties of a C*-algebra are defined in both the category of all C*-algebras and *-homomorphisms and  the category of all unital C*-algebras and unital *-homomorphisms.

\begin{definition} A $C^*$-algebra $A$ is {\bf projective} if for any C*-algebra $B$ and any ideal $I$ of $B$, any $\ast$-homomorphism
$\varphi: A \to B/I$ lifts to a $\ast$-homomorphism $\overline{\varphi}: A\to B$ with $\varphi = \pi\circ \overline{\varphi}.$ In
diagrammic notation: $$ \xymatrix{&B\ar@{->>}[d]^{\pi}\\
A\ar@{-->}[ur]^{\overline\varphi}\ar[r]^{\varphi} &B/I}$$
\end{definition}

\begin{definition}  A $C^*$-algebra $A$ is {\bf semiprojective}  if for any $C^*$-algebra $A$, any increasing chain of ideals $I_n$
in $A$  and every $\ast$-homomorphism $\varphi:A \to B/I$, where   $I=\overline{\bigcup I_n}$, there exist
 $n\in \mathbb N$ and  a $\ast$-homomorphism $\overline \varphi : A\to B/I_n$ making the following diagram commute:
$$\xymatrix{ & B \ar@{->>}[d]\\& B/I_n \ar@{->>}[d] \\ A\ar@{-->}[ru]^{\overline \varphi}\ar[r]^{\varphi} & B/I}$$
\end{definition}

\bigskip

For C*-algebras $B_n$, $n\in \mathbb N$, let us consider the following two C*-algebras of sequences with entries in $B_n$, $n\in \mathbb N$:
$$\prod B_n = \{(b_n)\;|\; b_n \in B_n, \; \sup \|b_n\| < \infty\},$$
$$\bigoplus B_n = \{(b_n)\in \prod B_n\;|\; \lim \|b_n\| = 0\}.$$ Clearly $\bigoplus B_n$ is  an ideal in $\prod B_n$.

\begin{definition} Let $\mathcal B$ be a class of C*-algebras. A C*-algebra $A$ is {\bf weakly semiprojective with respect to $\mathcal B$} if for any sequence $B_n\in \mathcal B$, any $\ast$-homomorphism  $\varphi: A\to \prod B_k/\bigoplus B_k$ lifts to a $\ast$-homomorphism  $\overline{\varphi}: A\to \prod B_k$ with $\varphi = \pi\circ \overline{\varphi}.$
\end{definition}
If the class $\mathcal B$ above is the class of all $C^*$-algebras then we say that $A$ is {\bf weakly semiprojective}.

\begin{definition}\label{def msp}
A separable $C^*$-algebra $A$ is {\bf matricially semiprojective} if one of the following equivalent conditions holds:
\begin{enumerate}
\item For any sequence $k_n$, any $\ast$-homomorphism
$$\varphi\colon A\to\prod M_{k_n}(\C) / \bigoplus M_{k_n}(\C)$$
lifts to a $\ast$-homomorphism $\overline{\varphi}\colon A\to\prod M_{k_n}(\C)$ with $\pi\circ\overline{\varphi}=\varphi$, i.e.\ $A$ is weakly semiprojective with respect to the class of matrix algebras.
\item For any sequence $F_n$ of finite-dimensional $C^*$-algebras, any $\ast$-homomorphism
$$\varphi\colon A\to\prod F_n / \bigoplus F_n$$
lifts to a $\ast$-homomorphism $\overline{\varphi}\colon A\to\prod F_n$ with $\pi\circ\overline{\varphi}=\varphi$, i.e.\ $A$ is weakly semiprojective with respect to the class of finite-dimensional $C^*$-algebras.
\end{enumerate}
\end{definition}

\begin{remark} For commutative C*-algebras matricial semiprojectivity is equivalent to the condition that any $\ast$-homomorphism $\varphi\colon A\to\prod M_{n}(\C) / \bigoplus M_{n}(\C)$ lifts. Indeed since a commutative C*-algebra has
a 1-dimensional representation and therefore finite-dimensional representations of any dimension, any $\ast$-homomorphism
$\varphi\colon A\to\prod M_{k_n}(\C) / \bigoplus M_{k_n}(\C)$ defines a $\ast$-homomorphism  to
$\prod M_{n}(\C) / \bigoplus M_{n}(\C)$ and the statement follows easily.
\end{remark}

Of course projectivity implies semiprojectivity, semiprojectivity implies weak semiprojectivity, and weak semiprojectivity implies matricial semiprojectivity.

The following reformulation of matricial semiprojectivity from \cite{Lor97} will be used several times throughout the paper.

\begin{proposition}(\cite[19.1.3]{Lor97})\label{ReformulationMWSP} A $C^*$-algebra $A$ is matricially semiprojective if  and only if for any sequence $k_n$, any $\ast$-homomorphism
$$\varphi\colon A\to\prod M_{k_n}(\C) / \bigoplus M_{k_n}(\C),$$
any finite subset $F\subset A$ and $\varepsilon>0$, there exists a $\ast$-homomorphism $\overline{\varphi}\colon A\to\prod M_{k_n}(\C)$ with
$$\|(\pi\circ\overline{\varphi})(a)-\varphi(a)\|\leq\varepsilon$$
for all $a\in F$.
\end{proposition}

\begin{definition} A $\ast$-homomorphism $\alpha: A\to B$ is {\bf matricially semiprojective} if one of the following equivalent conditions holds:
\begin{enumerate}
\item For any sequence $k_n$ and any $\ast$-homomorphism
$\varphi\colon B\to\prod M_{k_n}(\C) / \bigoplus M_{k_n}(\C)$, the $\ast$-homomorphism $\varphi\circ \alpha$ lifts
to a $\ast$-homomorphism $\overline{\varphi}\colon A\to\prod M_{k_n}(\C)$ with $\pi\circ\overline{\varphi}=\varphi\circ \alpha$.
\item For any sequence $F_n$ of finite-dimensional $C^*$-algebras and any $\ast$-homomorphism
$\varphi\colon B\to\prod F_n / \bigoplus F_n$, the $\ast$-homomorphism $\varphi\circ \alpha$
lifts to a $\ast$-homomorphism $\overline{\varphi}\colon A\to\prod F_n$ with $\pi\circ\overline{\varphi}=\varphi\circ \alpha.$
\end{enumerate}
\end{definition}

It is not hard to show that a non-unital C*-algebra is matricially semiprojective (in the non-unital category) if and only if its unitization is matricially semiprojective in the unital category.  The next proposition shows that for considering matricial semiprojectivity of a unital commutative C*-algebra it does not matter whether we consider  the category of all C*-algebras or the category of all unital C*-algebras.

\begin{proposition} Let $A$ be a unital commutative C*-algebra. The following are equivalent:

(i) $A$ is matricially semiprojective in the category of all C*-algebras and *-homomorphisms,

(ii) $A$ is matricially semiprojective in the category of all unital C*-algebras and unital *-homomorphisms.

\end{proposition}
\begin{proof} (i) $\Rightarrow$ (ii): Let $\phi: A \to \prod M_{k_n}(\C) / \bigoplus M_{k_n}(\C)$ be a unital *-homomorphism. It lifts to a *-homomorphism $\psi = (\psi_n): A \to \prod M_{k_n}(\C)$. Then $\psi(1) = (p_1, p_2, \ldots)$ is a sequence of projections such that $\|p_n-1_{k_n}\|\to 0$ and therefore (see e.g. \cite[Prop. 2.2.4 and 2.2.7]{Rordam}) $p_n= 1_{k_n}$, for sufficiently large $n$, say $n>N$. For $n\le N$, let $\tilde \psi_n: A \to M_{k_n} $ be any unital *-homomorphism (it exists, since $A$ is commutative and therefore has 1-dimensional representations). Replacing $\psi_n$ by $\tilde \psi_n$ when $n\le N$ and keeping $\psi_n$ for $n>N$ we obtain a unital lift of $\psi$.

(ii) $\Rightarrow$ (i):  Let $\phi: A \to \prod M_{k_n}(\C) / \bigoplus M_{k_n}(\C)$ be a *-homomorphism. By the well-known stability of projections  (\cite[p. 32]{Rordam}),  the projection $p=\phi(1)$ can be represented by a sequence $(p_1, p_2, \ldots)$ of projections. Then  $\phi$ can be considered as a unital *-homomorphism from $A$ to $\prod p_nM_{k_n}(\C)p_n / \bigoplus p_nM_{k_n}(\C)p_n$ which by Proposition \ref{ReformulationMWSP} lifts to a unital *-homomorphism from $A$ to $\prod p_nM_{k_n}(\C)p_n$. Combining it with the inclusion $\prod p_nM_{k_n}(\C)p_n \subset \prod M_{k_n}(\C)$ we obtain a lift of $\phi$.
\end{proof}

Therefore from now on we can  and will restrict to the unital category.

\subsection{Matricial semiprojectivity of 2-dimensional NCCW complexes}

The class of $n$-dimensional {\it noncommutative CW  (NCCW) complexes}  is defined recursively by saying that the $0$-dimensional ones are exactly the finite-dimensional $C^*$-algebras, and that the $n$-dimensional ones are pullbacks of the form
\[
A_n=C({\mathbb I}^n,F_n)\oplus_{C(\partial{\mathbb I}^n,F_n)}A_{n-1}
\]
where $A_{n-1}$ is an $(n-1)$-dimensional NCCW, $F_n$ is finite-dimensional, and the pullback is taken over the canonical map from $C({\mathbb I}^n,F_n)$ to $C(\partial{\mathbb I}^n,F_n)$ on one hand, and an arbitrary  unital $*$-homomorphism $\gamma_n \colon A_{n-1}\to C(\partial{\mathbb I}^n,F_n)$ on the other.

 Recall that an element $[x]$ of the ordered $K$-group $K_0(A)$ of a unital $C^*$-algebra $A$ is called an {\it infinitesimal}  if
\[
-[1_A] \le n[x] \le [1_A]
\]
for any $n\in \mathbb Z$. Any $K_0$-map induced by a unital $*$-homomorphism will map infinitesimals to infinitesimals.

\begin{example}\label{bott}  Let $S^2$ denote the 2-dimensional sphere, and $\mathrm{Bott}_n$  be the projection in $M_2(C_0(\mathbb R^2)^{+}) = M_2(C(S^2))$ given by
$$\mathrm{Bott}_n = \left(\begin{array}{cc} |z^n|^2 & z^n(1-|z^n|^2)^{1/2}\\ \bar z^n(1-|x^n|^2)^{1/2} & 1- |z^n|^2\end{array}\right),$$ where $\mathbb R^2$ is identified with $\{z: \;|z|<1\}$.
It is well known that
$$K_0(C(S^2)) = \mathbb Z\cdot\bott \oplus \,\mathbb Z\cdot[1]$$
where $\bott = \mathrm{Bott}_1 - [1].$
Moreover, the positive cone is given by
$$K_0(C(S^2))_+ = \{(n, m)\;|\; m>0\}\cup \{(0, 0)\},$$
i.e.\ the order is the strict order from the second summand. It follows that the set of infinitesimals equals $\{(n, 0)\;|\; n\in \mathbb Z\}$.
\end{example}

Eilers, Loring and Pedersen proved  that any 2-dimensional NCCW complex which has only torsion infinitesimals in its ordered $K_0$-group is matricially semiprojective (\cite[8.2.2(ii)]{ELP}).

\subsection{Covering dimension and finite generation of commutative C*-algebras}

The following proposition is well-known.

\begin{proposition}\label{CovDimFinGen} Let $X$ be  a compact metrizable space.  Then $C(X)$ is finitely generated if and only if $\dim X < \infty$.
\end{proposition}
\begin{proof} Suppose $\dim X < \infty$. Then $X$ embeds into $\mathbb R^{2n+1}$ \cite[Th. 50.5]{Munkres}. Hence we can write $x\in X$ as $x = (x_1, \ldots, x_{2n+1})$ and we define $f_i\in C(X)$ by $f_i(x) =x_i$, $i= 1, \ldots, 2n+1$. It follows from Stone-Weierstrass theorem that the unit function and $f_i$, $i= 1, \ldots, 2n+1$, generate $C(X)$.

Now suppose that $C(X)$ is finitely generated. Let $f_1, \ldots, f_n\in C(X)$ be  a generating set. Then the map
$$x\mapsto (f_1(x), \ldots, f_n(x))$$ is an embedding of $X$ into $\mathbb R^n$. By \cite[Th. 50.6]{Munkres}, $\dim X \le n$.
\end{proof}


\section{Proofs of auxiliary topological results}
\subsection{A theorem on extending mappings to spheres}

As is well known, a compact metrizable space
$X$ satisfies the inequality $\dim X > n$ if and only if there exists a  closed subspace
$A$ of the space X and a continuous mapping $f: A \to S^n$ which cannot be extended to $X$ (\cite[Th. 1.9.3]{Engelking}).
We will need a similar statement but with the additional requirement that $\dim A \le n$.

\medskip

 Recall that a {\it shrinking} of the cover $\{A_s\}_{s\in S}$ of a topological
space $X$ is any cover $\{B_s\}_{s\in S}$ of the space $X$ such that $B_s \subset A_s$, for
every $s\in S$. A shrinking is open (closed) if all its members are open
(closed) subsets of the space $X$.

\medskip

Let $X$ be a topological space and $A , B$ a pair of disjoint
subsets of the space $X$. A set $L \subseteq X$ is a {\it partition between $A$
and $B$} if there exist open sets $U , W \subseteq X$ satisfying
$$A \subset U,\; B \subset W, \; U \cap W = \emptyset, \; X\setminus L = U \cup W.$$

\medskip

 We will need the following  notions of dimension which for compact metrizable spaces coincide.  For a subset $U$ of a topological space $X$, let $\partial U$ denote its boundary.

\medskip

The {\it large inductive dimension} of $X$ is the smallest $n$, denoted by $\operatorname{Ind} X$,  such that  for every closed set $A\subset X$ and each open set $V\subset X$ which contains $A$ there exists an open set $U\subset X$ such that $A \subset U\subset V$ and $\operatorname{Ind} \partial U \le n-1$ (\cite[Def. 1.6.1]{Engelking}).

Let $X$ be a set and $\mathcal A$ be  a family of subsets of $X$. The {\it order} of $\mathcal A$, denoted by $ord \;\mathcal A$, is the largest $n$ such that $\mathcal A$ contains $(n+1)$ sets with a non-empty intersection (\cite[Def. 1.6.6]{Engelking}). The {\it covering dimension} of $X$, denoted by $\dim \; X$, is the smallest $n$ such that any finite open cover of $X$ has a finite open refinement of order $n$ (\cite[Def. 1.6.7]{Engelking}).

\medskip

In this paper $X$ will always be a compact metrizable space.

 \begin{lemma}\label{Topology1} Suppose $\dim X \le n+1$. Then every finite open cover $\{U_i\}$ of $X$
has

(i) a closed shrinking $\{B_i\}$ such that $\dim \partial B_i \le n$, for each $i$,

(ii) an open shrinking $\{V_i\}$ such that  $\dim \partial V_i \le n$, for each $i$.
\end{lemma}
\begin{proof} (i): We will prove it by induction. At first suppose we have  a 2-element cover,
$$X = U_1\bigcup U_2.$$ Then $$(X\setminus U_1)\cap(X\setminus U_2) = \emptyset$$ and $X\setminus U_i$ is closed, $i=1, 2.$
Since for compact metrizable spaces  the large inductive dimension coincides with the covering dimension ((\cite[Th. 1.6.11]{Engelking})), there exist
open subsets $V_i \supset (X\setminus U_i)$, $i=1,2$, such that $V_1\cap V_2 = \emptyset$ and $\dim \partial V_i \le n$, $i=1, 2$.
Let $B_i = X\setminus V_i$. Then $\{B_i\}$ is a closed shrinking of $\{U_i\}$ and since $ \partial B_i =  \partial V_i$, we have $\dim \partial B_i \le n$, $i=1,2$.

Suppose the statement is true for any k-element open cover of $X$. Consider a cover $\{U_i\}_{i=1}^{k+1}$. Let $\tilde U_1 = U_1\cup U_2$. Then for the cover $\tilde U_1, U_3, \ldots, U_{k+1}$
there exists a closed shrinking $\{B_i\}$ such that $\dim \partial B_i \le n$, for each $i$. Since $B_1\bigcap U_1$ and $B_1\cap U_2$ form an open cover of $B_1$ and since $\dim B_1 \le n+1$, there exists
its closed shrinking $G_i\supset B_1\cap U_i$, $i=1, 2$, such that $\dim \partial G_i \le n$, for $i=1, 2$. It follows that $G_1, G_2, B_3, \ldots, B_{k+1}$ is a required shrinking.

(ii) Let $\{C_i\}$ be any closed shrinking of $\{U_i\}$. For each $C_i$, there is an open subset $W_i$ such that $$U_i\supset W_i \supset C_i$$ and $\dim \partial W_i \le n$. Hence $\{W_i\}$ is  a required shrinking.
\end{proof}

\begin{lemma}\label{partition} Let $A$ and $B$ be two closed disjoint subsets of $X$. Let $L$ be a partition between $\partial A$ and $\partial B$. Then there exists a partition $\tilde L$ between  $A$ and $B$
such that $\tilde L\subseteq L$. Moreover as $\tilde L$ one can take $L\setminus\left(L\cap\left(A\cup B\right)\right)$.
\end{lemma}
\begin{proof} Since $L$ is a partition between $\partial A$ and $\partial B$, there are open subsets $U$ and $V$ such that
$$\partial A \subset U, \;\; \partial B \subset V, \;\; U\cap V = \emptyset, \;\; L = X\setminus (U\cup V).$$ Let
$$U' = \left(U\setminus\left(U\cap B\right)\right)\cup \left(V\cap A\right), \; \; V' = \left(V\setminus\left(V\cap A\right)\right)\cup \left(U\cap B\right).$$
Since $V\cap \partial A = \emptyset$, the set $V\cap A$ is open. Hence $U'$ is open. Similarly $V'$ is open. We notice also that
\begin{equation}\label{eq1}U'\supset \partial A, \;\; V'\supset \partial B,\end{equation}

\begin{equation}\label{eq2}U'\cap V' = \emptyset,\;\; U'\cap B = \emptyset, \;\; V'\cap A = \emptyset\end{equation} and
\begin{equation}\label{eq3}X\setminus \left(U'\cup V'\right) = X\setminus \left(U\cup V\right) = L.\end{equation} Let
$$\tilde U = U'\cup A, \; \; \tilde V = V'\cup B.$$  By (\ref{eq1}), $\tilde U$ and $\tilde V$ are open. By (\ref{eq2}), $\tilde U \cap \tilde V = \emptyset$. Let
$$\tilde L = X\setminus \left(\tilde U \cup\tilde V\right).$$ Then $\tilde L$ is a partition between $A$ and $B$ and it follows from (\ref{eq3}) that $\tilde L \supset L$.
\end{proof}

\begin{lemma}\label{Topology2} Suppose $\dim X = n+1$. Then there exists a sequence $$(A_1, B_1), \ldots, (A_{n+1}, B_{n+1})$$ of $n+1$ pairs of disjoint
closed subsets of $X$ satisfying $\dim A_i \le n$, $\dim B_i\le n$, $i=1, \ldots, n+1$, such that for any partition $L_i$ between $A_i$ and $B_i$, $i=1, \ldots, n+1$, $$\bigcap_{i=1}^{n+1} L_i \neq \emptyset.$$
\end{lemma}
\begin{proof} According to the proof of Th. 1.7.9 in \cite{Engelking}, $\dim X = n+1$ implies that for any $(n+2)$-element cover $\{U_i\}$ and any its closed shrinking $\{M_i\}$, the sequence of $(n+1)$ pairs $(N_i = X\setminus U_i, \;M_i)_{i=1}^{n+1}$ has the property that for any partition $P_i$ between $M_i$ and $N_i$, $i=1, \ldots, n+1$, $$\bigcap_{i=1}^{n+1} P_i \neq \emptyset.$$ Now let $\{U_i\}_{i=1}^{n+2}$ be any $(n+2)$-element open cover of $X$. By Lemma \ref{Topology1} we can find its open shrinking $\{V_i\}_{i=1}^{n+2}$ such that $\dim \partial V_i \le n$ for each $i$ and a closed shrinking $\{M_i\}_{i=1}^{n+2}$ of $\{V_i\}_{i=1}^{n+2}$ such that $\dim \partial M_i \le n$ for each $i$.
       Let $A_i = \partial N_i$, $B_i =
\partial M_i$, $i=1, \ldots, n+1$. Then $\dim A_i \le n$, $\dim B_i \le n$ for each $i$. Let $L_i$ be a partition between $A_i$ and $B_i$, for each $i$. By Lemma \ref{partition} there exists a partition $\tilde L_i$
between $N_i$ and $M_i$ such that $\tilde L_i \subseteq L_i$. Hence   $$\bigcap_{i=1}^{n+1} L_i \supseteq \bigcap_{i=1}^{n+1} \tilde L_i \neq \emptyset.\qedhere $$
\end{proof}

\begin{theorem}\label{Topology3} Suppose $\dim X = n+1$. Then there exists a closed subset $A$ of $X$ with $\dim A \le n$ and a continuous map $f: A \to S^n$ which cannot be extended to $X$.
\end{theorem}
\begin{proof} Let $A_i, B_i$, $i=1, \ldots, n+1$, be as in Lemma \ref{Topology2}. Let $$ A = \bigcup_{i=1}^{n+1} \left(A_i\cup B_i\right).$$ Since all $A_i$'s and $B_i$'s are closed, by Th. 1.5.3 in \cite{Engelking}, $\dim A \le n$.
The proof of Th.1.9.3 in \cite{Engelking} shows that there exists a continuous map $f: A\to S^n$ which cannot be extended to $X$.
\end{proof}


\subsection{Hopf's extension theorem for rational cohomology}\label{sec:Hopf}
We will need an analogue of Hopf's Extension Theorem when integral cohomology is replaced by the rational one.
By cohomology we mean \v Cech cohomology. All necessary information on \v Cech cohomology can be found in \cite{Nag83}.

\medskip

For any continuous mapping $f$ we will denote by $f^*$ the map induced by $f$ in cohomology.

Let $R$ be  a space, $C$ its closed subspace and $G$ an abelian group. Let $h^*: H^n(R, G) \to H^n(C, G)$ be the homomorphism induced by the embedding of $C$ into $R$.
An element $e\in H^n(C, G)$ is called {\it extendible over $R$}  if it is the image of some element $\tilde e\in H^n(R, G)$ under $h^*$ \cite{Nag83}.

\begin{theorem}\label{rational} Let $R$ be a compact metrizable space of $dim R \le n+1$ and $C$ a closed subset of $R$. Suppose $f$ is a continuous mapping of $C$ into $S^n$. Then $f$ can be extended to a continuous mapping $F$ of $R$ into $S^n$ if and only if $f^*e\in H^n(C, \mathbb Q)$ is extendible over $R$ for every element $e\in H^n(S^n, \mathbb Q)$.
\end{theorem}
\begin{proof} There is only one step in the proof of Hopf's Extension Theorem  (Th. VIII.1 in \cite{Nag83}) which requires a modification for the rational case. Namely
the proof uses a statement, called the statement $E_n$ in \cite[page 232]{Nag83}:

\medskip

 Let $L$ be a subcomplex of an $(n+1)$-complex $K$ and suppose $f$ is a simplicial mapping of $L$ into $S^n$. For a fixed positively oriented simplex $y_0^n$ of $S^n$ and $n$-simplices $x_{\alpha}^n$ of $L$ we define $f_0^{\alpha}$ as \begin{equation}\label{DefiningCoefficients}f_0^{\alpha} = \left\{
     \begin{array}{lcl}
       1 &;&  f(x_{\alpha}^n) = y_0^n,\\
       &&\\
       -1 &;& f(x_{\alpha}^n) = -y_0^n,\\&&\\
       0 &;& \text{otherwise}.
     \end{array}
   \right.
  \end{equation} If the $n$-cocycle $\phi^n = f^*y_0^n = \sum_{\alpha\in A}  f_0^{\alpha} x_{\alpha}^n$ of $L$ is the restriction of a cocycle of $K$, then $f$ can be extended to a mapping $F$ of $K$ into $S^n$ which is continuous on the closure of every simplex of $K$.

\medskip
\noindent Moreover for the proof of Hopf's Extension Theorem for a compact metrizable space one can assume  that the complex $K$ in the statement $E_n$ is finite, since
all the nerves (which are used in the definition of \v Cech cohomology) are finite complexes.

The proof of our theorem  will be completed after we prove the rational analogue of the statement $E_n$. We will do it in a separate lemma.
\end{proof}

\begin{lemma}
Let $L$ be a  subcomplex of a finite (n+1)-complex $K$ and suppose $f$ is a simplicial mapping of $L$ into $S^n$. For a fixed positively oriented simplex $y_0^n$ of $S^n$ and $n$-simplices $x_{\alpha}^n$ of $L$ we define $f_0^{\alpha}$ as in (\ref{DefiningCoefficients}).  If the $n$-cocycle $\phi^n = f^*y_0^n = \sum_{\alpha\in A}  f_0^{\alpha} x_{\alpha}^n$ of $L$ is the restriction of a rational cocycle of $K$, then $f$ can be extended to a mapping $F$ of $K$ into $S^n$ which is continuous on the closure of every simplex of $K$.
\end{lemma}
\begin{proof} Let  $$d_f(K) = \sum_{\alpha} f_0^{\alpha},$$ where the sum is taken over all positively oriented simplices $x_{\alpha}^n$ of $K$.
In \cite{Nag83} there was proved a statement (not involving cohomology) called $D_n$:

Let either of $K$ and $S^n$ be an oriented $n$-sphere or one of its succesive barycentric subdivisions. Suppose $f$ is a simplicial mapping of $K$ into $S^n$. If $d_f(K) =0$, then $f$ is homotopic to $0$.

The statement $E_n$ in \cite{Nag83} was deduced from $D_n$. We also will deduce our lemma from $D_n$. For that we will make only minor modification in the  proof of $D_n \Rightarrow E_n$ in \cite{Nag83}. We will describe only this modification to not copy the whole proof of $D_n \Rightarrow E_n$. Since $\phi^n$ is the restriction of a rational cocycle $\psi^n = \sum_{\beta} g^{\beta} x_{\beta}^n$, we have that all $g^{\beta}$ are rational (in contrast to the proof of $D_n\rightarrow E_n$, where they are integers). Since $K$ is a finite complex, the sum is finite, so we can find $N\in \mathbb N$  such that all $Ng^{\beta}$ are integers. In the same way as in the proof of
 $D_n\rightarrow E_n$, assuming without loss of generality that $g^{\beta} \ge 0$ and using successive barycentric subdivisions,  we will construct $Ng^{\beta}$ subsimplices $x^n_{\beta k}$ (in contrast to the proof of $D_n\rightarrow E_n$, where they construct $g^{\beta}$ such subsimplices). The rest of the proof goes without change.
\end{proof}

\subsection{Subspaces with nonvanishing cohomology}\label{sec:subsets}

The following theorem will be crucial for proving that matricial semiprojectivity for $C(X)$ forces a bound on the covering dimension of $X$.

\begin{theorem}\label{Topology5}  Let $X$ be a compact metrizable space with $\dim X < \infty$. Then for every integer $n$ with $0\le n<\dim X$, there exists a closed subset $A$ of $X$ such that $\dim A=n$ and $H^n(A, \mathbb Q)\neq 0$.
\end{theorem}
\begin{proof} We can assume that $\dim X = n+1$. By Theorem \ref{Topology3} there exists a closed subset $A$ of $X$ with $\dim A \le n$ and a continuous mapping $f: A \to S^n$ which cannot be extended to $X$. By
Theorem \ref{rational} there is a non-extendible element in $H^n(A, \mathbb Q).$ Hence the map $$h^*:H^n(X, \mathbb Q)\to H^n(A, \mathbb Q)$$ is not onto (here $h$ is used for the inclusion $A \subset X$). In particular $H^n(A, \mathbb Q)\neq 0$.
Since $H^n(A)\neq 0$ implies that $\dim A \ge n$, we have $\dim A =n$.
\end{proof}


\section{Matricial semiprojectivity for commutative $C^*$-algebras}\label{sec:main}

In this section we give a proof of our main result.

\medskip

\begin{theorem}\label{MainTheorem} ({\bf Main theorem.}) Let $X$ be a compact metrizable space with $\dim X < \infty$. Then $C(X)$ is matricially semiprojective if and only if $\dim X \leq 2$ and $H^2(X; \Q) = 0$.
\end{theorem}

The sufficiency of the above criterion is proved in Theorem \ref{sufficiency}, the necessity is shown in Theorem \ref{necessity}.
It is not clear whether the finiteness assumption on the dimension of $X$ is actually necessary.  In fact it can be dropped in many cases. Indeed we prove that if  a space $X$ has a closed subspace $Y$ s.t. $C(Y)$ is not matricially semiprojective, then $C(X)$ is not matricially semiprojective as well (Lemma \ref{Lemma1}). This holds for arbitrary $X$, not necessarily finite-dimensional. To prove the necessity of the dimensional restriction $\dim X \leq 2$ in the main theorem (Th. \ref{MainTheorem}) we then look for closed subspaces failing matricial semiprojectivity. Here the assumption that $X$ is finite-dimensional  is only needed to guarantee the existence of  closed subsets of smaller dimension (which is then used to guarantee the existence of  closed subsets of smaller dimension with prescribed cohomology (Theorem \ref{Topology5})). However, there do exist infinite-dimensional compact spaces all whose finite-dimensional closed subsets are 0-dimensional (e.g., \cite{Hen67},\cite{Wal79}), and we do not know whether the methods developed here can be adjusted to cover these examples as well. We therefore ask

\begin{question} Can the assumption $\dim X<\infty$ be dropped?
\end{question}

\medskip


\subsection{Sufficiency}\label{sec:suff}

\begin{lemma}\label{MatrSPMap}
Let $X$ be a finite CW-complex of dimension at most two, $Y$ a compact metrizable space,  $\alpha\colon C(X) \to C(Y)$  a unital $\ast$-homomorphism with $H^2(\alpha^*;\mathbb Q) = 0$. Then $\alpha$ is matricially semiprojective.
\end{lemma}

\begin{proof}
Theorem 8.1.1 and Corollary 8.2.2(ii) in \cite{ELP} show that $\alpha$ is matricially semiprojective provided that a certain subgroup $G_0$ of $K_0(C(X))$ is mapped to the torsion subgroup of $K_0(C(Y))$ under $K_0(\alpha)$. We will identify $G_0$ and show that the assumption on $\alpha$ forces $K_0(\alpha)(G_0)$ to be torsion.

Following \cite{ELP}, the subgroup $G_0$ is defined in the following way: write
$$C(X)=C(\D,\C^n)\oplus_{C(S^1,\C^n)} C(X^{(1)})$$
as a 2-dimensional NCCW (with $X^{(1)}$ the 1-skeleton of $X$), then $G_0$ is given as $\ker(\rho_r)_*$ where $\rho_r\colon C(X) \to C(X^{(1)})$ denotes the projection onto the right summand.
By the 6-term exact sequence in K-theory, $G_0$ thus coincides with the image of $K_0(i)$, where
$$i=\id\oplus\, 0\colon C_0(\R^2,\C^n)\to C(X),$$
and hence comes from the inclusion of all 2-cells into $X$.
Using $K_0(C_0(\R^2))=H^2(\R^2)$ (see e.g. \cite[Theorem B.7]{HannesThesis}), we find the image of $G_0$ under the Chern character $\ch\colon K_0(C(X))\to H^0(X;\Q)\oplus H^2(X;\Q)$ to be contained in $H^2(X;\Q)$.
But then
$$\ch(K_0(\alpha)(G_0))=H^2(\alpha^*;\Q)(\ch(G_0))=0,$$
which implies that $K_0(\alpha)(G_0)$ is torsion as $\ch$ is rationally an isomorphism.
\end{proof}

\begin{lemma}\label{lem:extension}
Let $X, Y$ be compact spaces and $\alpha\colon C(X)\to C(Y)$ a unital, injective $\ast$-homomorphism.
Then any $\ast$-homomorphism $\varphi\colon C(X)\to M_n(\C)$ extends to a $\ast$-homomorphism $\overline{\varphi}$ making
\[
\xymatrix{C(X) \ar[r]^\varphi \ar@{^{(}->}[d]_\alpha & M_n(\C) \\
C(Y) \ar@{..>}[ur]_{\overline{\varphi}}}
\]
commute.
\end{lemma}

\begin{proof}
Since $\varphi$ is unitarily equivalent to a sum of irreducible representations, we may assume that $\varphi=\ev_{x_1}\oplus...\oplus\ev_{x_n}$ for some points $x_1,...,x_n\in X$.
Using surjectivity of the dual map $\alpha^*\colon Y\to X$, we choose preimages $y_1,...,y_n\in Y$ and set $\overline{\varphi}:=\ev_{y_1}\oplus...\oplus\ev_{y_n}$ which then satisfies $\overline{\varphi}\circ\alpha=\varphi$.
\end{proof}

\begin{lemma}\label{Lemma2}
Let
\[
C(X)= \varinjlim (C(X_i), \theta_i^{i+1})
\]
be separable and suppose all connecting maps $\theta_i^{i+1}\colon C(X_i) \to C(X_{i+1})$ are injective.
Then the following holds: if all $ \theta_i^{\infty}\colon C(X_i) \to C(X)$ are matricially semiprojective, then $C(X)$ is matricially semiprojective.
\end{lemma}

\begin{proof}
Let $\phi\colon C(X) \to \prod M_{d_i}(\C)/\bigoplus M_{d_i}(\C)$ be a lifting problem for $C(X)$.
To prove that $\phi$ is liftable, by Proposition \ref{ReformulationMWSP} it is sufficient to show that for any finite subset $\mathcal F$ of $C(X)$ and any $\varepsilon>0$ there exists a $\ast$-homomorphism $\psi\colon C(X) \to \prod M_{d_i}(\C)$ such that $\|\pi\circ \psi (g) - \phi(g)\| \le \varepsilon$ for all $g\in \mathcal F$.
So let a finite subset $\mathcal F\subset C(X)$ and $\varepsilon >0$ be given.
There is $N$ and $f_g\in C(X_N)$, for each $g\in \mathcal F$, such that
\begin{equation}\label{2InductiveLimit'}\|g - \theta_N^{\infty}(f_g)\|\le \varepsilon/2.\end{equation}
By assumption, $\phi\circ \theta_N^{\infty}$ lifts to some $\ast$-homomorphism $\overline{\phi_N}:C(X_N)\to \prod M_{d_i}(\C)$.
\[
\xymatrix{& C(X_N) \ar[r]^{\overline{\phi_N}} \ar@{^{(}->}[d]^{\theta_N^{\infty}} & \prod M_{d_i}(\C) \ar@{->>}[d]^\pi\\
& C(X) \ar@{-->}[ur]^{\psi} \ar[r]^(0.3)\phi & \prod M_{d_i}(\C)/\bigoplus M_{d_i}(\C) }
\]
Now each coordinate $C(X_N)\to M_{d_i}(\C)$ of $\overline{\phi_N}$ extends to $C(X)$ by Lemma \ref{lem:extension} and hence so does $\overline{\phi_N}$, i.e.\ we obtain a $\ast$-homomorphism $\psi$ with $\psi\circ\theta_N^\infty=\overline{\phi_N}$ as indicated above.
Now for any $g\in \mathcal F$ one verifies
\begin{eqnarray*}
\|(\pi\circ \psi)(g)-\phi(g)\| &\leq& \|(\pi\circ\psi)(\theta_N^\infty(f_g))-\phi(\theta_N^\infty(f_g))\| +\varepsilon\\
&=& \|\pi((\psi\circ\theta_N^\infty)(f_g)-\overline{\phi_N}(f_g))\|+\varepsilon \\
&=&\varepsilon. \qedhere
\end{eqnarray*}
\end{proof}

 In what follows we will use several times the following result of Freudenthal \cite[ p. 183]{Fr37} (see also  \cite[ footnote 1 on p. 278]{Mardesic}).

\begin{lemma}\label{Fr} (Freudenthal \cite{Fr37})
Let $X$ be  a compact metrizable space with  $\dim(X) \leq n$. Then $X$ can be written as an inverse limit of $n$-dimensional finite CW-complexes  where the connecting maps can be chosen to be surjective.
\end{lemma}

\begin{theorem}\label{sufficiency}
Let $X$ be a compact metrizable space such that $\dim(X) \leq 2$ and $H^2(X; \mathbb Q) =0$. Then $C(X)$ is matricially semiprojective.
\end{theorem}

\begin{proof}
Since $\dim(X) \leq 2$, by Lemma \ref{Fr} we can  write $X$ as an inverse limit of 2-dimensional finite CW-complexes $X_i$ with surjective connecting maps. The induced connecting maps in $C(X)= \varinjlim C(X_i)$ are therefore injective. The corresponding maps from $C(X_i)$ to $C(X)$ are  matricially semiprojective by Lemma \ref{MatrSPMap}. The claim now follows from Lemma \ref{Lemma2}.
\end{proof}

The following proposition  shows that Theorem \ref{sufficiency} fails for arbitrary compact Hausdorff spaces.

 \begin{proposition} There exists a $0$-dimensional compact Hausdorff (non-metrizable) space $X$ such that $C(X)$ is not matricially semiprojective.
 \end{proposition}
 \begin{proof} We start with constructing continuum of almost commuting projections in $\prod M_n$  which are not close to commuting projections in $\prod M_n$ (we  learnt this construction from private communication with D. Hadwin, M. Rordam and P. S{\o}rensen). Let $E\subset \prod M_n$ be the set of all projections $(P_n)$ such that $\operatorname{rank} P_n =1$, $n\in \mathbb N$. Let $\pi: \prod M_n \to \prod M_n /\oplus M_n$ be the canonical surjection and let $f$ be a maximal family of commuting projections in $\pi(E)$. We will show that $f$ does not lift. Assume, for the sake of contradiction, that it lifts to a family $F$ of commuting projections in $\prod M_n$. Then for each $n\in \mathbb N$ there is an orthonormal basis in $\mathbb C^n$ with respect to which all projections in $F$ are given by diagonal matrices. Let
 $Q = (Q_n)\in \prod M_n$ be the projection given by matrices $Q_n$, $n\in \mathbb N$,  whose all entries are equal to $1/n$ with respect to the corresponding basis. Then $Q$ almost commutes with all projections in $F$. Hence $\pi(Q)$ commutes with $f$. However $\pi(Q)\notin f$, since for any diagonal projection $Q_n'\in M_n$ one has  $\|Q_n - Q_n'\|\ge |1/n - 1| \nrightarrow 0.$
 This contradicts to maximality of $f$.

 Now let $A$ be the unital universal C*-algebra generated by continuum of commuting projections. Then  $A =C(X),$ for some compact Hausdorff totally disconnected space $X$. Since $X$ is compact and totally disconnected, $\operatorname{ind} X = 0$ (\cite[Th. 1.4.5]{Engelking}). It follows that   $\dim X  =0$ (\cite[Cor.2.2 + Prop.2.3]{Pears}). Therefore the conditions of Theorem \ref{sufficiency} are fulfilled for $X$, but $C(X)$ is not matricially semiprojective.
\end{proof}

\medskip

\subsection{Necessity}\label{sec:nec}
Below we use the abbreviation c.p.c.\ for completely positive contractive maps.
We use complete positivity as a very convenient tool to lift $\ast$-homomorphisms at hand to linear maps. In what follows all one needs to know about c.p.c.\ maps is that each $\ast$-homomorphism is   c.p.c.\ and that a  c.p.c.\ map  from a commutative (more generally, from a nuclear) C*-algebra to any quotient C*-algebra lifts to a c.p.c.\ map  by Choi-Effros theorem \cite{ChoiEffros} (see also \cite[Th. C.3]{BrownOzawa}).
More  information on c.p.c.\ maps can be found in \cite[section 1.5]{BrownOzawa}.

\medskip

Let $A, B$ be C*-algebras, $G\subseteq A$, $\delta>0$. We say that a map $\phi:A \to B$ is {\it $(G, \delta)$-multiplicative} if
$$\|\phi(a_1a_2)-\phi(a_1)\phi(a_2)\|\le \delta,$$ for any $a_1, a_2\in G$.

\medskip

We will need a few lemmas.

\begin{lemma}\label{lem:msp quantitative}
Let $A$ be a separable, nuclear $C^*$-algebra. The following are equivalent:
\begin{enumerate}
\item $A$ is matricially semiprojective.
\item For every finite subset $F\subset A$ and $\varepsilon>0$ there exist a finite subset $G\subset A$ and $\delta>0$ such that the following holds:
For every $n$ and every $(G,\delta)$-multiplicative c.p.c.\ map $\varphi\colon A\to M_n(\C)$ there exists a $\ast$-homomorphism $\psi\colon A\to M_n(\C)$ with
\[
\|\varphi(x)-\psi(x)\|\leq\varepsilon \quad \text{for all $x\in F$}.
\]
\end{enumerate}
\end{lemma}

\begin{proof}
$(1)\Rightarrow(2)$: Assume $(2)$ is false, i.e.\ there exists a finite subset $F\subset A$ and $\varepsilon>0$ such that there is a increasing sequence of finite subsets $G_n$ with dense union in $A$ and $(G_n,1/n)$-multiplicative c.p.c.\ maps $ \varphi\colon A\to M_{k_n}(\C)$ for some $k_n$ such that no $\varphi_n$ is close (up to $\varepsilon$ on $F$) to a $\ast$-homomorphism.
We then obtain a c.p.c.\ map $\varphi=(\varphi_n)\colon A\to \prod M_{k_n}(\C)$ which drops to a $\ast$-homomorphism $\pi\circ\varphi$ after dividing out $\bigoplus M_{k_n}(\C)$.
Since $A$ is matricially semiprojective, there exists a $\ast$-homomorphism $\psi=(\psi_n)\colon A\to\prod M_{k_n}(\C)$ which lifts $\pi\circ\varphi$.
But this means $\|\psi_n(x)-\varphi_n(x)\|\to 0$ for any $x\in A$, contradicting the assumption that the $\varphi_n$'s could not be approximated on $F$ by $\ast$-homomorphisms.

$(2)\Rightarrow(1)$: Let $(F_m)_m$ be an increasing sequence of finite subsets with dense union in $A$ and choose $(\G_m,\delta_m)$ as in $(2)$ according to $(F_m,\varepsilon_m=1/m)$.
Now if $\varphi\colon A\to \prod M_{k_n}(\C) / \bigoplus M_{k_n}(\C)$ is a lifting problem for $A$, we first apply the Choi-Effros lifting theorem \cite{ChoiEffros} to obtain a c.p.c.\ lift $\overline{\varphi}=(\overline{\varphi}_n)\colon A\to\prod M_{k_n}(\C)$ of $\varphi$.
Since $\pi\circ\overline{\varphi}=\varphi$ is a $\ast$-homomorphism, the coordinate c.p.c.\ maps $\overline{\varphi_n}$ will be $(G_m,\delta_m)$-multiplicative for all $n\geq N_m$ for some $N_m$.
By $(2)$, there now exist $\ast$-homomorphisms $\psi_n\colon A\to M_{k_n}(\C)$ with $\max_{x\in F_m}\|\psi_n(x)-\overline{\varphi}_n(x)\|\leq 1/m$ for $N_m\leq n<N_{m+1}$.
It follows that $\|\psi_n(x)-\overline{\varphi}_n(x)\|\to 0$ on a dense set, i.e.\ the resulting $\ast$-homomorphism $\psi=(\psi_n)$ lifts $\varphi$.
\end{proof}

\begin{lemma}\label{1}
Let $A$ be separable, nuclear and matricially semiprojective, then the following holds:
For every finite subset $F\subset A$ and every $\varepsilon >0$ there exists a finite subset $G\subset A$ and $\delta >0$, so that for any sequence $(k_n)_n$ and any $(G,\delta)$-multiplicative c.p.c.\ map
$$\varphi\colon A\to \prod M_{k_n}(\mathbb C)/\bigoplus M_{k_n}(\mathbb C)
$$
there exists a $\ast$-homomorphism $\psi\colon A\to\prod M_{k_n}(\C)$ with
\[
\|(\pi\circ\psi)(x)-\varphi(x)\|<\varepsilon
\]
for all $x\in F$.
\end{lemma}

\begin{proof}
Choose $G,\delta$ as in Lemma \ref{lem:msp quantitative} and let a $(G,\delta/2)$-multiplicative c.p.c.\ map $\varphi$ as in the statement be given.
We use the Choi-Effros theorem \cite{ChoiEffros} to lift $\varphi$ to a c.p.c.\ map $\overline{\varphi}=(\overline{\varphi}_n)$ from $A$ to $\prod M_{k_n}(\mathbb C)$.
The coordinates $\overline{\varphi}_n$ will then be $(G,\delta)$-multiplicative for all large $n$, hence  by Lemma \ref{lem:msp quantitative} there exist $\ast$-homomorphisms $\psi_n\colon A\to M_{k_n}(\C)$ which agree up to $\varepsilon$ with $\overline{\varphi}_n$ on $F$. Then $\psi=(\psi_n)_n$ (with $\psi_n=0$ for small $n$) is the desired $\ast$-homomorphism satisfying
$$\|(\pi\circ\psi)(x)-\varphi(x)\|=\limsup_n \|\psi_n(x)-\overline{\varphi}_n(x)\|\leq\varepsilon$$
for all $x\in F$.
\end{proof}

For definition of cancellation property mentioned below see \cite[Def. 7.3.1]{Rordam}.

\begin{lemma}\label{2}
Let $(A_n,\theta_n^{n+1})$ be a unital inductive system of $C^*$-algebras and consider the canonical inclusion

\begin{eqnarray*}
\iota\colon & \varinjlim A_n  \to  \prod A_n / \bigoplus A_n ,\\
& \theta_N^\infty(a)  \mapsto  [\theta_N^n(a)]
\end{eqnarray*}
then the following holds:
if all $C^*$-algebras $A_n$ have cancellation, then the induced map $K_0(\iota)$ is injective.
\end{lemma}

\begin{proof}
We may assume that $A:=\varinjlim A_n$ is unital.
Let an element $x\in K_0(A)$ be given, then there are $N_0,k$ and projections $p,q$ in $M_k(A_{N_0})$ with
$$x=(\theta_{N_0}^\infty)_*([p]-[q]).$$
Now if $\iota_*(x)=0$, we in particular find
$$[\theta_{N_0}^n(p)]=[\theta_{N_0}^n(q)]$$
in $K_0(A_n)$ for all large $n$. Hence there is $N_1$ and, by cancellation of $A_{N_1}$, a partial isometry $v\in M_{2k}(A_{N_1})$ such that
$$\theta_{N_0}^{N_1}(p)=v^*v\sim vv^*=\theta_{N_0}^{N_1}(q).$$
Consequently, we find $\theta_{N_0}^n(p)\sim\theta_{N_0}^n(q)$ via $\theta_{N_1}^n(v)$ for all large $n$, i.e.\ $\theta_{N_0}^\infty(p)\sim\theta_{N_0}^\infty$ in $A$ and hence $x=0$.
\end{proof}

\begin{lemma}\label{added}
Let $(A_n,\theta_n^{n+1})$ be an inductive system of $C^*$-algebras such that $A = \varinjlim A_n$ is nuclear. Let $n\in \mathbb N$,   $F\subset A_n$ and $G\subset A$ be finite sets,  $\varepsilon >0$, $\delta >0$.  Then there are $m\geq n$ and a $(G,\delta)$-multiplicative c.p.c.\ map $\sigma\colon A\to A_m$ with
$$\|(\sigma\circ\varphi_m^\infty)(a)-a\|<\varepsilon$$
for all $a\in \varphi_n^m(F)$.
\end{lemma}
\begin{proof}
Since $A$ is nuclear, by Choi-Effros theorem \cite{ChoiEffros} the canonical inclusion $$A \to \varinjlim_m A_m \subset \prod A_m/\bigoplus A_m$$ lifts to a c.p.c.\ map $\gamma = \{\gamma_m\}_{m\in \mathbb N}: A \to \prod A_m$. It implies that for all sufficiently large $m$, $\gamma_m$ is $(G, \delta)$-multiplicative and also that for any $f\in A_n$ $$\lim_{m\to \infty} \|\gamma_m(\varphi_n^{\infty}(f))- \varphi_n^m(f)\| =0.$$ In particular  there is $m>n$ such that $\gamma_m$ is $(G, \delta)$-multiplicative and $$\|\gamma_m(\varphi_n^{\infty}(f) - \varphi_n^m(f)\| < \varepsilon,$$ for all $f\in F$.
Set $\sigma = \gamma_m$.
\end{proof}

The following two results are well known. We include their proofs for reader's convenience.

\begin{lemma}\label{3}
If A is a $C^*$-algebra with finitely generated K-theory, then there exists a finite subset $F\subset A$ and some $\varepsilon>0$ such that any two $\ast$-homomorphisms which agree up to $\varepsilon$ on $F$ induce the same morphisms on K-theory.
\end{lemma}
\begin{proof} The statement is a consequence of the fact that if the difference between 2 projections has norm smaller than 1, then the projections are unitarily equivalent (\cite[Prop. 2.2.4 and 2.2.6]{Rordam}). Indeed, let $[p_1]-[q_1], \ldots, [p_N]-[q_N]$ generate $K_0(A)$, where $p_l\in M_{n_l}(\C)$, $q_l\in M_{k_l}(\C).$ Let $F$ be the set of all entries of all the projections $p_l, q_l$.
Let $$\epsilon = \frac{1}{2}\max\{\max_{1\le l \le N} \frac{1}{n_l^2}, \max_{1\le l \le N} \frac{1}{k_l^2}\}.$$ If two $\ast$-homomorphisms $\phi$ and $\psi$ agree up to $\epsilon$ on $F$, then for each $l=1, \ldots, N$ we have
\begin{multline*} \|\phi \otimes \id_{M_{n_l}(\C)}(p_l) - \psi \otimes \id_{M_{n_l}(\C)}(p_l)\| = \|\left(\phi((p_l)_{ij}) - \psi((p_l)_{ij})\right)_{i, j= 1}^{n_l}\|\\
\le \sum_{i, j=1}^{n_l}\| \phi((p_l)_{ij}) - \psi((p_l)_{ij})\| \le n_l^2\epsilon < 1 \end{multline*} and similarly $$\|\phi \otimes \id_{M_{k_l}(\C)}(q_l) - \psi \otimes \id_{M_{k_l}(\C)}(q_l)\|  <1.$$ Hence
 $\phi \otimes \id_{M_{n_l}(\C)}(p_l)$ and $\psi \otimes \id_{M_{n_l}(\C)}(p_l)$ are unitarily equivalent, and $\phi \otimes \id_{M_{k_l}(\C)}(q_l)$ and   $\psi \otimes \id_{M_{k_l}(\C)}(q_l)$ are unitarily equivalent. It follows that the  morphisms  induced by $\phi$ and $\psi$ in K-theory coincide on $[p_1]-[q_1], \ldots, [p_N]-[q_N]$ and hence are the same.
 \end{proof}

\begin{lemma}\label{AfterReport}
Let $F_k$ be a finite-dimensional C*-algebra, $k\in  \mathbb N$. Then  there are no infinitesimals in $K_0(\prod F_k)$.
\end{lemma}
\begin{proof} It is easy to see that there are no infinitesimals in the $K_0$ of the matrix algebras.  Since the image of an infinitesimal under any homomorphism (in particular under the projection map) is an infinitesimal, it follows that in $K_0(F_k)$ and $\prod K_0(F_k)$  there are no infinitesimals as well.  As $K_0(\prod F_k)$ is contained in  $\prod K_0(F_k)$ (\cite[p. 199]{Lor88}), the result follows.
\end{proof}

Now we are ready to prove the necessity of the conditions in the main theorem.

\begin{theorem}\label{NecessityFor2-dim}
If $X$ is compact metrizable, 2-dimensional space with $H^2(X;\mathbb{Q})\neq 0$, then $C(X)$ is not matricially semiprojective.
\end{theorem}

\begin{proof}
By contradiction.
By Lemma \ref{Fr}, we write $X$ as a surjective inverse limit of finite 2-dimensional CW-complexes $X_n$ to obtain $C(X)=\varinjlim C(X_n)$ with injective connecting homomorphisms $\varphi_n^m$.
We shall proceed in several steps:

\begin{enumerate}
\item As \v Cech cohomology is a continuous functor and $H^2(X,\Q)\neq 0$ by assumption, there exists a non-torsion element $x\in H^2(X_n)$ such that $(\varphi_n^\infty)_*(x)$ is a non-torsion element in $H^2(X)$.

\item Let $\bott\in K_0(S^2)$ denote the Bott element (see Example \ref{bott}).
We claim that there exists a $\ast$-homomorphism $\varrho\colon C(S^2)\to C(X_n)$ with $$K_0(\varrho)(\bott)=x.$$
Indeed, according to \cite[Th. VIII 2 on p.149]{HW48} there is a 1-1-correspondence between homotopy classes of continuous maps $X\to S^2$ and the group $H^2(X)$ which identifies a map $\alpha$ with the image of 1 under $$H^2(\alpha)\colon \Z=H^2(S^2)\to H^2(X).$$
Since under the identification $K_0(S^2)=H^0(S^2)\oplus H^2(S^2)$ the Bott element generates $H^2(S^2)$, this shows that we can find a suitable $\alpha$ which induces a homomorphism $\alpha^*$ with the desired property.
\item We apply Lemma \ref{3} to the $C^*$-algebra $C(S^2)$, let a finite subset $F\subset C(S^2)$ and $\varepsilon>0$ be accordingly.
\item Choose $G\subset C(X),\delta>0$ with respect to $(\varphi_n^\infty\circ\varrho)(F)$ and $\varepsilon/2$ as in Lemma \ref{1}.

\item By Lemma \ref{added} we can find $m\geq n$ and a $(G,\delta)$-multiplicative c.p.c.\ map $\sigma\colon C(X)\to C(X_m)$ with
$$\|(\sigma\circ\varphi_m^\infty)(a)-a\|<\varepsilon/2$$
for all $a\in (\varphi_n^m\circ\varrho)(F)$.

\item We use Theorem 1.1 of \cite{DL92} to find a unital AF-algebra $A=\varinjlim F_k$ (with finite-dimensional $F_k$'s) and a unital embedding $\phi\colon C(X_m)\to A$ which is rationally an isomorphism on $K_0$.
By Lemma \ref{2} we may regard $\phi$ as a map from $A$ to $\prod F_{k}/\bigoplus F_{k}$ which is still rationally injective on $K_0$.
This in particular implies that $$\phi_{\ast}\left((\varphi_n^m)_{\ast}(x)\right)\neq 0.$$

\item Since $(\phi\circ\sigma)$ is $(G,\delta)$-multiplicative, we may apply Lemma \ref{1} to find a
$\ast$-homomorphism $\phi'\colon C(X)\to\prod F_{k}/\bigoplus F_{k}$ with $$\|\phi'(a)-(\phi\circ\sigma)(a)\|<\varepsilon/2$$ for all $a\in(\varphi_n^m\circ\varrho)(F)$.
Now by matricial semiprojectivity of $C(X)$, $\phi'$ lifts to a homomorphism $\overline{\phi}$ :

\[
\xymatrix{& C(X) \ar@{..>}[r]^(.4){\overline{\phi}} \ar@/^1pc/[d]^\sigma \ar@{-->}[dr]^(.37){\phi'} & \prod F_k \ar@{->>}[d]^\pi\\
& C(X_m) \ar[u]^{\varphi_m^\infty} \ar[r]^(0.4)\phi & \prod F_k/\bigoplus F_k \\
C(S^2) \ar[r]^\varrho & C(X_n) \ar[u]^{\varphi_n^m}}
\]

\item It follows that the two homomorphisms $\phi'\circ\varphi_n^\infty\circ\varrho$ and $\phi\circ\varphi_n^m\circ\varrho$ agree up to $\varepsilon$ on $F$ and hence induce the same $K_0$-map by step (3).
However, the first map factors through $\prod F_{k}$ and therefore kills all infinitesimals in $K_0(C(S^2))$ (because there are no infinitesimals in $K_0(\prod F_k)$  by Lemma \ref{AfterReport}), in particular the element $\bott$ (cf. Example \ref{bott}).
The second map, on the other hand, does not vanish on $\bott$ by construction. Contradiction.\qedhere
\end{enumerate}
\end{proof}

\begin{lemma}\label{Lemma1}
Let $X$ be a compact metrizable space. Then the $C^*$-algebra $C(X)$ is matricially semiprojective if and only if $C(Y)$ is matricially semiprojective for every closed subset $Y \subseteq X$.
\end{lemma}
\begin{proof}
``If'' is obvious.
For the other implication we fix a metric $d$ on $X$, let $Y \subseteq X$ be a closed subset and assume that $C(X)$ is matricially semiprojective.
By Proposition \ref{ReformulationMWSP}, it will be sufficient to prove that for any finite subset $\G \subset C(X)$ and $\varepsilon >0$, any lifting problem $\varphi\colon C(Y) \to \prod M_{n_k}(\C)/\bigoplus M_{n_k}(\C)$ admits a solution up to $\varepsilon$ on $\G_{|Y}=\{g_{|Y}\colon g\in \G\}$.
By compactness of $X$, the elements of $\G$ are uniformly continuous, i.e.\ $(\varepsilon,\delta)$-continuous for some $\delta>0$ (meaning $d(x, x')<\delta$ implies $|g(x) - g(x')|<\varepsilon$ for all $g\in \G$).
For $r\geq 0$, we consider the compact neighborhood $Y_r = \{x\colon d(x,Y) \leq r\}$ of $Y$ in $X$ and denote by $\rho_{Y_r}^X\colon C(X) \to C(Y_r)$ the restriction map.

We claim that if we solve the lifting problem $\varphi\circ\rho_Y^X$ (which we can by assumption), we can also solve the lifting problem $\varphi\circ\rho_Y^{Y_\delta}$.
Indeed, let $\overline\varphi$ be any lift of $\varphi\circ \rho^X_Y$, then it suffices to show that its coordinates $\overline\varphi_{k}\colon C(X) \to M_{n_k}(\C)$ factor through $\varrho^X_{Y_\delta}$ for all sufficiently large $k$. For this consider the element $h\in C(X)$ given by
$$h(x)=d(x,Y)$$
which is strictly positive for $\ker(\varrho^X_Y)$.
We have $\|\overline{\varphi}_k(h)\|<\delta$ for large $k$ as $\pi(\overline{\varphi}(h))=0$, hence also $\overline{\varphi}_k((h-\delta)_+)=0$ for large $k$.
But as $(h-\delta)_+$ is strictly positive for $\ker(\varrho^X_{Y_\delta})$, the claim follows.

Now let $\overline\varphi_{\delta}$ be any lift of $\varphi\circ \rho^{Y_\delta}_Y$.
\[
\xymatrix{
& C(X) \ar[r]^(.45){\overline{\varphi}} \ar@{->>}[d]^(.4){\rho_{Y_{\delta}}^X}  & \prod M_{n_k}(\C) \ar@{->>}[dd]^\pi\\
& C(Y_{\delta}) \ar[ur]^{\overline\varphi_{\delta}} \ar@{->>}[d]^(.4){\rho_{Y}^{Y_{\delta}}} & \\
& C(Y) \ar[r]^(.3){\varphi} \ar@{..>}[uur]^{\psi} & \prod M_{n_k}(\C)/\bigoplus M_{n_k}(\C)}
\]
Each coordinate $(\overline\varphi_{\delta})_k\colon C(Y_{\delta}) \to M_{n_k}(\C)$ is unitarily equivalent to a direct sum of point evaluations $\ev_y$ with $y\in Y_{\delta}$.
Replacing each occurrence of $\ev_y$ by $\ev_{y'}$ for some $y'\in Y$ with $d(y, y')\leq \delta$, we obtain a $\ast$-homomorphism $\psi\colon C(Y) \to \prod_k M_{n_k}(\C)$.
Furthermore, by the choice of $\delta$ we have ensured that
$$\|\psi\circ \rho^{Y_{\delta}}_Y(g_{|Y_{\delta}}) - (\overline\varphi_{\delta})(g_{|Y_{\delta}})\| <\varepsilon$$
holds for all $g \in \G$. Therefore we find for all $g\in \G$
$$\|(\pi\circ\psi)(g_{|Y}) - \varphi(g_{|Y})\|<\| (\pi\circ\overline\varphi_{\delta})(g_{|Y_{\delta}})-(\varphi\circ\varrho^{Y_\delta}_Y)(g_{|Y_\delta})\|+\varepsilon=\varepsilon. \qedhere$$
\end{proof}

\begin{theorem}\label{necessity}
Let $X$ be a compact metrizable space with $\dim X < \infty$.
If $C(X)$ is matricially semiprojective, then $\dim(X)\leq 2$ and $H^2(X;\Q)=0$.
\end{theorem}

\begin{proof}
We first show that $\dim(X) \leq 2$.
Assume for contradiction that $\dim X > 2$, then by Theorem \ref{Topology5} there exists a closed subset $Y\subseteq X$ such that $\dim Y=2$ and $H^2(Y; \mathbb Q)\neq 0$.
Then $C(Y)$ is not matricially semiprojective by Theorem \ref{NecessityFor2-dim} and hence neither is $C(X)$ by Lemma \ref{Lemma1}, a contradiction.
Thus $\dim(X)\leq 2$, and by Theorem \ref{NecessityFor2-dim} we conclude $H^2(X;\mathbb Q)=0$.
\end{proof}

\section{Further applications}\label{sec:app}
\subsection{Liftings from the Calkin algebra}\label{sec:calkin}

\begin{theorem}\label{CommCategory} For any separable $C^*$-algebra  $A$, any $\ast$-homomorphism   from $A$ to $\ell_{\infty}/c_0$ lifts to a $\ast$-homomorphism  from $A$ to $\ell_{\infty}$.
\end{theorem}
\begin{proof} Let $\tau: A \to \ell_{\infty}/c_0$ be a  $\ast$-homomorphism. Let $\pi: \ell_{\infty}\to \ell_{\infty}/c_0$ be the canonical surjection,  $B = \tau(A)$ and $D = \pi^{-1}(B)$.
Let $\hat D$ denote the maximal ideal space of $D$ (that is the set of all multiplicative functionals on $D$ endowed with the $\ast$-weak topology).
Then $\hat D = \hat B \bigcup \hat c_0$ and $\hat c_0 = \mathbb N$ consists of the coordinate functionals $e_n$, $n\in \mathbb N$,
defined by  $e_n(x) = x_n$, for any $x= (x_n)_{n\in \mathbb N}\in \ell_{\infty}$. Since $D$ is separable, $\hat D$ is metrizable. We claim that $d(e_n, \hat B)\to 0$ as $n\to \infty$. Indeed, otherwise there would be a  subsequence $n_k$ and $C>0$ such that $d(e_{n_k}, \hat B) > C$, for all $k$. By Banach-Alaoglu theorem there would be a subsequence $n_{k_l}$ such that $e_{n_{k_l}}$ converges $\ast$-weakly to some $e$ which obviously belongs to $\hat B$. We obtain  $d(e, \hat B)\ge C$, a contradiction.

Therefore there are $g_n\in \hat B$ such that $d(e_n, g_n)\to 0$.
We define  a $\ast$-homomorphism $s: A \to \ell_{\infty}$ by $$s(a) = (g_n(x))_{n\in \mathbb N},$$  $a\in A$,  where $x$ is any preimage of $\tau(a)$. As $g_n \in \hat B$, $s$ is well-defined. Since $d(e_n, g_n)\to 0$, for any $a\in A$ we have  $$s(a) - x =
(g_n(x))_{n\in \mathbb N} - (e_n(x))_{n\in \mathbb N}\in c_0.$$  Hence $\pi(s(a)) = \tau(a)$, for any $a$, so $s$ is a lift of $\tau$.
\end{proof}

\begin{remark} Theorem \ref{CommCategory} states that in the category of separable C*-algebras  any morphism   to $\ell_{\infty}/c_0$ lifts. In fact the same holds for the categories of separable Banach spaces and separable Banach algebras.  It is not hard to modify the arguments above slightly to give a unified proof for all 3 cases.
\end{remark}

\begin{remark}\label{MWSPCommCategory}   In the category of commutative $C^*$-algebras  the notion of matricial semiprojectivity is obtained  by replacing matrix algebras by their maximal abelian selfadjoint subalgebras, that is, by the algebras of diagonal matrices relative to some basis. Thus in the commutative category matricial semiprojectivity reduces to lifting $\ast$-homomorphisms to $\ell_{\infty}/c_0$, which is automatic by Theorem \ref{CommCategory}.
\end{remark}

\begin{corollary}\label{new}  Let $Y$ be a closed subset of a compact metrizable space $X$ and let $\phi: C(Y)\to Q(H)$. Let $r: C(X) \to C(Y)$ be the
 restriction map. If $\phi\circ r$ lifts, then $\phi$ lifts.
\end{corollary}
\begin{proof} Let $\tilde \psi$ be a lift of $\phi\circ r$. By a corollary of Voiculescu's theorem (\cite[Cor. II.5.9]{Davidson}), $\tilde \psi$ is approximately unitarily equivalent to direct
sum of irreducible representations of $C(X)$, which are evaluations at points of $X$. In particular it implies that there is an orthonormal basis $\{e_k\}$ of $H$ and points $x_k\in X$, $k\in \mathbb N$, such that $$\tilde \psi (f) - \operatorname{diag} (f(x_k))\in K(H),$$
for all $f\in C(X)$, where by  $\operatorname{diag} (f(x_k))$ we mean the diagonal operator relative to the basis $\{e_k\}$ with the diagonal entries $f(x_k)$, $k\in
\mathbb N$. Define $\psi: C(X) \to B(H)$ by $$\psi(f) = \operatorname{diag} (f(x_k)).$$ Then $\psi$ is a lift of $\phi\circ r$ and its range is contained in the set of diagonal operators relative to the basis $\{e_i\}$, which can be identified with $\ell_{\infty}$. The image in the Calkin algebra of the set of diagonal operators  can be identified with $\ell_{\infty}/c_0$. Thus we can think of $\phi\circ r$, and hence of $\phi$ because it has the same range as $\phi\circ r$,  as of a $\ast$-homomorphism to $\ell_{\infty}/c_0$. By Theorem \ref{CommCategory} $\phi$ lifts (to a $\ast$-homomorphism with range  consisting of diagonal operators relative to the basis $\{e_k\}$).
\end{proof}

\begin{remark} The preceding corollary could alternatively be obtained by arguments similar to those of Lemma \ref{Lemma1} combined with the corollary of Voiculescu's theorem.
The other way around, Lemma \ref{Lemma1} could also be deduced from Theorem \ref{CommCategory}.
\end{remark}

\medskip

Let $\{d_n\}$ be a sequence of natural numbers. Let us fix an orthonormal basis in $H$.
We define an embedding $$i: \prod M_{d_n} \to B(H)$$ by sending the sequence $(x_n)_{n\in \mathbb N}$ of matrices to the corresponding block-diagonal operator, relative to the basis we fixed.  Define an embedding $$j: \prod M_{d_n}/\bigoplus M_{d_n} \to Q(H)$$ as $$j\left((x_n)_{n\in \mathbb N} + \bigoplus M_{d_n}\right) = i((x_n)_{n\in \mathbb N}) + K(H).$$
Then $$\pi\left(i\left(\prod M_{d_n}\right)\right) = j\left( \prod M_{d_n}/\bigoplus M_{d_n}\right).$$

\noindent We will say that a $\ast$-homomorphism $f: A\to Q(H)$ is a {\it limit of liftable $\ast$-homomorphisms} if it is a pointwise limit of $\ast$-homomorphisms $f_n: A \to Q(H)$ and all $f_n$'s are liftable to $B(H)$.

\begin{theorem}\label{ReformulationQ(H)-l-close} Let $X$ be a compact, metrizable space. The following are equivalent:

1) Each $\ast$-homomorphism from $C(X)$ to $Q(H)$ which is a limit of liftable $\ast$-homomorphisms, is liftable itself;

2) For any sequence $d_n$ of natural numbers and any $\ast$-homomorphism $\phi: C(X) \to \prod M_{d_n}/ \oplus M_{d_n}$, the $\ast$-homomorphism $j\circ \phi: C(X) \to Q(H)$ lifts.
This lifting property can be illustrated by the diagram $$\xymatrix {& & B(H) \ar[dd]^{\pi} \\ && \\ C(X) \ar[r]_{\phi \;\;\;\;\;\;\;\;\;\;} \ar@{.>}[uurr]& \prod M_{d_n}/ \oplus M_{d_n} \ar[r]_{\;\;\;\;\;\;\;\;\;j} & Q(H)}
$$

\medskip

\end{theorem}
\begin{proof} 1)$\Rightarrow$ 2): Let $\phi: C(X) \to \prod M_n/ \oplus M_n$.
Since $\pi^{-1}(j\circ\phi(C(X)) \subseteq i\left(\prod M_n\right) + K(H)$, we find $\pi^{-1}((j\circ\phi)(C(X)) $ to be a quasidiagonal set of operators. Define $$\psi: C(X)/Ker (j\circ\phi) \to Q(H)$$ by $$\psi (f+ Ker (j\circ\phi)) = (j\circ \phi) (f),$$ for any $f\in C(X)$. Since $\pi^{-1}\left(\psi\left(C(X)/Ker (j\circ\phi)\right)\right) = \pi^{-1}((j\circ\phi)(C(X)) $, $\psi$ is a quasidiagonal extension of a commutative $C^*$-algebra $C(X)/Ker (j\circ\phi)$. By \cite[Th. IX.8.2]{Davidson}, $\psi$ is a limit of liftable $\ast$-homomorphisms. Hence so is $j\circ \phi$. Hence $j\circ \phi$ is liftable.

2)$\Rightarrow$ 1): Suppose $\phi: C(X) \to Q(H)$ is a limit of liftable $\ast$-homomorphisms. By \cite[Th. IX.8.2]{Davidson}, whose proof holds for non-necessarily injective $\ast$-homomorphisms as well, we conclude that $\pi^{-1}(\phi(C(X))$ is a quasidiagonal family of operators. Therefore there is an increasing sequence of projections $P_n \uparrow 1$ such that $\|[P_n, T]\|\to 0$ for any $T\in \pi^{-1}(\phi(C(X))$. Let $\phi': C(X) \to B(H)$ be a c.p.c.\ lift of $\phi$ and let $a_1, a_2, \ldots$ be a dense subset in the unit ball of $C(X)$. Without loss of generality we can assume that
$$\|[P_n, \phi'(a_k)]\| < \frac{1}{2^n}, $$ when $n\ge k$. Then for any $k\in \mathbb N$
$$\sum(P_{n+1}-P_n)[P_{n+1}, \phi'(a_k)]\in K(H), $$
$$\sum (P_{n+1}-P_n)[P_{n}, \phi'(a_k)]\in K(H). $$
Define a map $\phi'': C(X) \to B(H)$  by $$\phi'' = \sum (P_{n+1}-P_n)\phi'(P_{n+1}-P_n)$$ (the sum here and all sums below converge with respect to the strong operator topology). For any $k\in \mathbb N$ we have  \\
\begin{alignat*}{1} \phi'(a_k) - \phi''(a_k) & = \sum (P_{n+1}-P_n)\phi'(a_k) - \sum (P_{n+1}-P_n)\phi'(a_k)(P_{n+1}-P_n) \\
= &\sum (P_{n+1}-P_n)^2\phi'(a_k) - \sum (P_{n+1}-P_n)\phi'(a_k)(P_{n+1}-P_n) \\
 = & \sum (P_{n+1}-P_n)[(P_{n+1}-P_n), \phi'(a_k)]\\
= & \sum (P_{n+1}-P_n)[P_{n+1}, \phi'(a_k)] - \sum (P_{n+1}-P_n)[P_{n}, \phi'(a_k)] \in K(H).
\end{alignat*}
\noindent Therefore $\phi''$ is a lift of $\phi$.  Let $d_n = \dim (P_{n+1} - P_n). $ Then $\phi= \pi\circ \phi''$ lands in $\prod M_{d_n}/ \oplus M_{d_n}$ and lifts by assumption.
\end{proof}

\begin{corollary}\label{MWSPimpliesQ-l-closed} If $\dim X \le 2$ and $H^2(X, \mathbb Q)=0$, then each $\ast$-homomorphism from $C(X)$ to $Q(H)$ which is a limit of liftable $\ast$-homomorphisms, is liftable itself.
\end{corollary}
\begin{proof} Suppose $\dim X \le 2$ and $H^2(X, \mathbb Q)=0$. By our main theorem  $C(X)$ is matricially semiprojective and hence satisfies the condition 2) of Theorem \ref{ReformulationQ(H)-l-close}.
Now the statement follows from
Theorem \ref{ReformulationQ(H)-l-close}.
\end{proof}

\begin{lemma}\label{InductiveLimit} Let $$C(X)= \lim_{\longrightarrow} (C(X_i), \theta_i^{i+1})$$ be the inductive limit of C*-algebras $C(X_i)$ with all the connecting maps $\theta_i^{i+1}: C(X_i) \to C(X_{i+1})$ being injective. Let $\phi: C(X) \to Q(H)$.  Then the following holds: if all $\phi \circ \theta_i^{\infty}: C(X_i) \to C(X)$ are liftable then $\phi$ is a limit of liftable $\ast$-homomorphisms.
\end{lemma}
\begin{proof} We need to show that for any finite subset $\mathcal F$ of $C(X)$ and any $\varepsilon >0$ there is a $\ast$-homomorphism $\psi: C(X) \to B(H)$ such that $\|\pi\circ \psi (g) - \phi(g)\| \le \varepsilon$ for all $g\in \mathcal F$.

Since $C(X)$ is an inductive limit of $C(X_i)$ with all the connecting maps $\theta_i^{i+1}: C(X_i) \to C(X_{i+1})$  being injective, there is $N$ and $f_g\in C(X_N)$, for each $g\in \mathcal F$, such that \begin{equation}\label{2InductiveLimit}\|g - \theta_i^{\infty}(f_g)\|\le \varepsilon/2.\end{equation} By the assumption $\phi\circ \theta_N^{\infty}$ lifts to some $\ast$-homomorphism $\overline{\phi_N}:C(X_N)\to B(H)$. By a corollary of Voiculescu's theorem (\cite[Cor. II.5.9]{Davidson}), $\overline{\phi_N}$ is approximately unitarily equivalent to a direct sum of irreducible representations, which are just evaluations at some points $z_k\in X_N$, $k\in\mathbb N$. Thus there is an orthonormal basis in $H$ such that for each $f\in C(X_N)$, $\overline{\phi_N}(f) - \oplus_k f(z_k) \in K(H)$, where $\oplus_k f(z_k)$ means the diagonal, relative to this basis, operator with eigenvalues $f(z_k)$, $k\in \mathbb N$. Hence \begin{equation}\label{1InductiveLimit}\pi( \overline{\phi_N}(f)) = \pi(\oplus_k f(z_k)).\end{equation} Let $\gamma_N: X\to X_N$ be a surjection such that $$\theta_N^{\infty}(f)(x) = f(\gamma_N(x)), $$ for all $f\in C(X_N)$, $x\in X$. For each $z_k$ fix some its preimage $x_k$ under the map $\gamma_N$, that is $\gamma_N(x_k) = z_k$. Define a $\ast$-homomorphism $\psi: C(X) \to B(H)$ by $$\psi(g) = \oplus g(x_k),$$ (meaning the diagonal operator relative to the same basis as before).
Then, by (\ref{2InductiveLimit}) and (\ref{1InductiveLimit}),  for any $g\in \mathcal F$ we have
\begin{alignat*}{1}& \| \pi(\psi(g)) - \phi(g)\| \le \| \pi(\psi(g)) - \pi( \overline{\phi_N}(f_g)\| + \| \pi( \overline{\phi_N}(f_g) - \phi(g)\|   \\
&= \| \pi(\psi(g)) - \pi(\oplus_k f_g(z_k))\| + \| \phi(\theta_N^{\infty}(f_g)) - \phi(g)\|  \\&= \| \pi(\oplus_k g(x_k)) - \pi(\oplus_k f_g(z_k))\| + \| \phi(\theta_N^{\infty}(f_g)) - \phi(g)\|   \\& \le   \| \oplus_k g(x_k) - \oplus_k f_g(z_k)\| + \| \theta_N^{\infty}(f_g) - g\| \\ & \le  \|\oplus_k (g - \theta_{N, \infty}(f_g))(x_k)\| + \| \theta_N^{\infty}(f_g) - g\| \le  \varepsilon/2 + \varepsilon/2 = \varepsilon. \qedhere
\end{alignat*}
\end{proof}

\begin{lemma}\label{bouquet} Let $X$ be the disjoint union of a finite number of finite bouquets of circles. Let $\phi: C(X) \to Q(H)$ be a $\ast$-homomorphism such that $K_1(\phi) = 0$.
Then $\phi$ is liftable.
\end{lemma}
\begin{proof} First we will prove the lemma for a finite bouquet of circles $X = \bigvee_i \mathbb T^{(i)}$ and unital $\phi$. Extending a continuous function on $\mathbb T^{(i)}$ by an appropriate constant on the other circles we obtain an embedding $C(\mathbb T^{(i)})\subseteq C(\bigvee_i \mathbb T^{(i)}).$ The base-point in the bouquet we will identify with $0$ via the parametrization of each copy of $\mathbb T$ by points of the unit interval.

By semiprojectivity of $C(X)$ \cite{Loring89} and Blackadar's homotopy lifting theorem \cite{Blackadar} the question of whether $\phi$ lifts or not depends only on the homotopy type of $\phi$. Define a $\ast$-homomorphism $\phi': C(X)\to Q(H)$ to be equal to $\phi$ on all $C(\mathbb T^{(i)})$'s on which $\phi$ is injective, and to be $\phi'(f) = f(0) 1$ on all $C(\mathbb T^{(i)})$'s on which $\phi$ is not injective. Then $\phi'$ is homotopic to $\phi$. Indeed  $\phi$ is not injective on $C(\mathbb T^{(i)})$ if and only if the unitary operator $U_i = \phi(\id^{(i)})$ has a hole in its spectrum. Such operator can be connected by a continuous path of unitary operators with the identity operator which gives us a homotopy between $\phi$ and $\phi'$. Hence $K_1(\phi')=0$. Let $X'$ be the bouquet of all circles $\mathbb T^{(i)}$ such that  $\phi$ is injective on $C(\mathbb T^{(i)})$. Restricting $\phi'$ to functions which are constant outside $X'$ we obtain an injective $\ast$-homomorphism $\phi'': C(X') \to Q(H)$. Then   $K_1(\phi'')$ is the restriction of $K_1(\phi')$ onto the summands of $K_1(C(X))$  corresponding to $X'$ inside $X$, and hence is $0$. Since a bouquet of circles is a planar set, $\operatorname{Ext}(X)=\Hom(K_1(C(X)),Z)$ by Brown-Douglas-Fillmore (e.g. see \cite[Th. IX.7.2]{Davidson}), which implies that $\phi''$ represents a trivial extension, i.e.\ it lifts. Being the composition of $\phi''$ and the restriction map $C(X)\to C(X')$, $\phi'$ also lifts.

Now suppose  that  $X$ is the disjoint union of a finite number of finite bouquets of circles, $X = \bigsqcup X_i.$ Then $C(X) = \bigoplus C(X_i).$ We will identify $C(X_i)$ with its copy in $\bigoplus C(X_i).$ Let $p_i = \phi(1_{C(X_i)}).$ As is well known, mutually orthogonal projections in $Q(H)$ lift to mutually orthogonal projections in $B(H)$, so we can lift $p_i$'s to mutually orthogonal projections $P_i\in B(H).$ Let $\phi_i: C(X_i) \to Q(P_iH)$ be the restriction of $\phi$ onto $C(X_i)$. The $\phi_i$ is unital and $K_1(\phi) = 0$ implies that $K_1(\phi_i)  = 0$, for all $i$. Hence $\phi_i$ lifts to a $\ast$-homomorphism $\psi_i$, for all $i$. Then $\psi: C(X) \to B(H)$ defined as $\psi(\oplus f_i) = \oplus \psi_i(f_i)$, for any $f_i \in C(X_i)$,  is a lift of $\phi$.
\end{proof}

\begin{theorem}\label{ApplMain} Let $X$ be a compact metrizable space and $\dim X \le 1$. The following are equivalent:

(1) All $\ast$-homomorphisms from $C(X)$ to $Q(H)$ are liftable;

(2) $\Hom(H^1(X, \mathbb Z), \mathbb Z) = 0$.
\end{theorem}
\begin{proof} (2) $\Rightarrow$ (1): Let $\phi: C(X) \to Q(H)$. We need to prove that $\phi$ is liftable.
By Lemma \ref{Fr}, we write $X$ as inverse limit of 1-dimensional finite CW-complexes $X_n$ with all the connecting maps being surjective. Then  $$C(X)= \lim_{\longrightarrow} (C(X_i), \theta_i^{i+1})$$ with  all the connecting maps $\theta_i^{i+1}: C(X_i) \to C(X_{i+1})$ being injective. Since each 1-dimensional finite CW-complex is homotopic to the disjoint union of a finite number of finite bouquets of circles (see e.g. \cite{Loring89}), there are $\ast$-homomorphisms $f_i: C(X_i) \to C(\bigvee_{\alpha=1}^{m_i}\mathbb T_{\alpha})$ and $g_i: C(\bigvee_{\alpha=1}^{m_i}\mathbb T_{\alpha}) \to C(X_i)$ such that $$f_i\circ g_i \sim_{hom}\id_{C(\bigvee_{\alpha=1}^{m_i}\mathbb T_{\alpha})}\; \;\text{and}\; \; g_i\circ f_i \sim_{hom}\id_{C(X_i)}.$$
Since for compact metrizable spaces of dimension not larger than 1, $H^1(X, \mathbb Z) = K^1(X)$ (e.g. see \cite{HannesThesis}), we have  $\Hom(K_1(C(X)), \mathbb Z) = 0$. Since $K_1(Q(H)) = \mathbb Z$, it implies that $K_1(\phi\circ \theta_i^{\infty}) = 0$ and hence $K_1(\phi\circ\theta_i^{\infty}\circ g_i) = 0.$ By Lemma \ref{bouquet}, $\phi\circ\theta_i^{\infty}\circ g_i$ is liftable. Hence $\phi\circ\theta_i^{\infty}\circ g_i\circ f_i$ is liftable too. Since $X_i$ is a CW-complex and $\dim X_i\le 1$, $C(X_i)$ is semiprojective \cite{Loring89, AdamHannes}. Since $\phi\circ \theta_i^{\infty}$ is homotopic to $\phi\circ\theta_i^{\infty}\circ g_i\circ f_i$, by Blackadar's Homotopy Lifting Theorem \cite{Blackadar} $\phi\circ \theta_i^{\infty}$ is liftable.
By Lemma \ref{InductiveLimit}, $\phi$ is a limit of liftable $\ast$-homomorphisms. Since $\dim X \le 1$, we have $H^2(X, \mathbb Q)=0$ and by Corollary \ref{MWSPimpliesQ-l-closed}, $\phi$ is liftable.

(1) $\Rightarrow$ (2):  Since $\dim X \le 1$, $X$ embeds into $\mathbb R^3$ and by \cite[6.4 (c)]{Davie} we conclude that  $\operatorname{Ext}(X) = \Hom(H^1(X, \mathbb Z), \mathbb Z)$. Hence if  $\Hom(H^1(X, \mathbb Z), \mathbb Z) \neq 0$, there exists a non-liftable (injective) $\ast$-homomorphism from $C(X)$ to $Q(H)$.
\end{proof}

\begin{remark} It seems reasonable that the condition $\dim X \le 1$ in Theorem \ref{ApplMain} is necessary, i.e.\ we expect that for $X$ of finite covering dimension all $\ast$-homomorphisms from $C(X)$ to $Q(H)$ are liftable if and only if $\dim X \le 1$ and $\Hom(H^1(X, \mathbb Z), \mathbb Z) = 0$.
The missing ingredient for a proof is an analogue of Theorem \ref{Topology5} with cohomologies replaced by $\Hom(H^n(\cdot, \mathbb Z), \mathbb Z)$. We therefore ask:
\end{remark}

\begin{question}  For a compact metrizable space $X$, does $\infty >\dim X > n$ imply that there exists a closed subset $Y \subseteq X$ with $\dim Y =n$ and $\Hom(H^n(Y, \mathbb Z), \mathbb Z) \neq 0$?
\end{question}

\medskip

We finish this section with a lifting result for normal elements in the Calkin algebra $Q(H)$.
The question of when a normal element of the Calkin algebra lifts to a normal operator in $B(H)$ has been completely resolved by Brown-Douglas-Fillmore theory.
However, as we learned in private communication with specialists in the field, the following question appears to be still open:\\

 {\it For which compact subsets $X\subset \mathbb{R}^2$ does the following hold: Every normal element of the Calkin algebra with spectrum contained in $X$ lifts to a normal element in $B(H)$?}  \\

We give a complete answer below.

\begin{proposition}\label{normals} Let $X$ be a compact subset of the plane. The following are equivalent:

(i) Any normal element of the Calkin algebra with spectrum contained in $X$ lifts to a normal operator;

(ii) Any normal element of the Calkin algebra with spectrum contained in $X$ lifts to a normal operator with the same spectrum;

(iii) $\dim X\le 1$ and $H^1(X, \mathbb Z)=0$.
\end{proposition}

\begin{proof}
$(i)\Rightarrow (ii)$: Let $a\in Q(H)$ be a normal element (with spectrum inside $X$). Define a $\ast$-homomorphism $\phi: C(\sigma(a)) \to Q(H)$ by $\phi(z)=a$ (here by $z$ we denote the identity function on the plane). By assumption $a$ lifts to a normal operator $A\in B(H)$.  Let $r: C(\sigma(A))\to C(\sigma(a))$ be the restriction map. Since $A$ is a lift of $a$, $\phi\circ r$ is liftable. By Corollary \ref{new} $\phi$ lifts to some $\ast$-homomorphism $\psi: C(\sigma(a)) \to B(H)$. Then $\psi(z)$ is a normal lift of $a$ and $\sigma(\psi(z)) = \sigma(a).$

$(ii)\Rightarrow (iii)$:  $(ii)$ clearly implies that all $\ast$-homomorphisms from $C(X)$ to  $Q(H)$ are liftable. Therefore if $\dim X \le 1$, then $(iii)$ follows from Theorem \ref{ApplMain}. It remains to show that $\dim X \neq 2$. Suppose by contradiction that $\dim X = 2$, then by Theorem \ref{Topology5} there is $Y\subset X$ such that $\dim Y = 1$ and $H^1(Y, \mathbb Z)\neq 0.$ Since for planar sets
the first cohomology group
is free abelian
(\cite[Th. IX. 7.1]{Davidson}), we conclude that $\Hom(H^1(Y, \mathbb Z), \mathbb Z) \neq  0$. By Theorem \ref{ApplMain} there is a non-liftable $\ast$-homomorphism $\phi: C(Y) \to Q(H)$. Let $r: C(X) \to C(Y)$ be the restriction map. By Corollary \ref{new} $\phi\circ r$ is non-liftable, a contradiction.

$(iii)\Rightarrow (i)$: Let $a\in Q(H)$ be a normal element with spectrum inside $X$. Define a $\ast$-homomorphism $\phi: C(X) \to Q(H)$ by $\phi(z)=a$. By Theorem \ref{ApplMain} it must lift to some $\ast$-homomorphism $\psi: C(X) \to B(H)$. Then $\psi(z)$ is a normal lift of $a$.
\end{proof}

\begin{remark}
The implication $(i)\Rightarrow (ii)$ above could alternatively be deduced using classification of essentially normal operators of Brown, Douglas and Fillmore.
\end{remark}

\subsection{Around Blackadar's $\ell$-closedness}\label{sec:lclosed}

In \cite{Blackadar} Blackadar introduced the following notion of $\ell$-closed and $\ell$-open $C^*$-algebras. For any $C^*$-algebra $B$ and any ideal $I$ in $B$, let $\Hom(A, B)$ denote the set of all $\ast$-homomorphisms from $A$ to $B$  and let $\Hom(A, B, I)$ denote the set of all  $\ast$-homomorphisms from $A$ to $B/I$ which are liftable.

\begin{definition} \cite{Blackadar} A $C^*$-algebra $A$ is {\it $\ell$-closed} ({\it $\ell$-open}) if for any $C^*$-algebra $B$ and any ideal $I$ in $B$, the set $\Hom(A, B, I)$ is closed (open) w.r.t. the topology of pointwise convergence in the set $\Hom(A, B/I)$.
\end{definition}

 In \cite{Blackadar} Blackadar writes: ``It seems reasonable that if $X$ is any ANR (absolute neighborhood retract),
then $C(X)$ is $\ell$-closed''. In particular, every finite CW-complex should have this property. We  will see below that this is actually not the case.

Note that Corollary \ref{MWSPimpliesQ-l-closed} provides us with sufficient conditions for a space $X$ so that $\Hom(C(X),B(H),K(H))$ is closed in $\Hom(C(X),Q(H))$, namely $\dim X\le 2 $ and $H^2(X, \mathbb Q)=0$.
At least for CW-complexes, we can also find necessary conditions as follows.

\begin{lemma}\label{subsetsQ(H)-l-closed}
Let $Y$ be a closed subset of a compact, metrizable space $X$. Then the following holds:
If $\Hom(C(X),B(H),K(H))$ is closed in $\Hom(C(X),Q(H))$, then so is $\Hom(C(Y),B(H),K(H))$ in $\Hom(C(Y),Q(H))$.
\end{lemma}
\begin{proof} Let $r: C(X) \to C(Y)$ be the
 restriction map. If $\phi: C(Y) \to Q(H)$ is a limit of liftable $\ast$-homomorphisms, then so is $\phi\circ r$. Since
$\Hom(C(X),B(H),K(H))$ is closed, $\phi\circ r$ is liftable. By Corollary \ref{new}, $\phi$ is liftable as well.
\end{proof}

\begin{proposition}
Let $X$ be a CW-complex. If  $\Hom(C(X),B(H),K(H))$ is closed, then $\dim X \le 3$.
\end{proposition}

\begin{proof} Let $K$ be the suspended solenoid from \cite[example IX.11.3]{Davidson}.
In \cite{Davie} it is shown that there exists an extension of $C(K)$  which is a limit of trivial  extensions, but which is not trivial itself. Thus $\Hom(C(K),B(H),K(H))$ is not closed in $\Hom(C(K),Q(H))$.
Now suppose, by contradiction, that $\dim X \ge 4$. Let $I$ denote the unit interval. Since $K \subset I^4$, we have $K\subset X$.
By Lemma \ref{subsetsQ(H)-l-closed} it follows that $\Hom(C(X),B(H),K(H))$ is not closed, a contradiction.
\end{proof}

\begin{corollary}\label{l-closed}
Let $X$ be a CW-complex. If $C(X)$ is $\ell$-closed, then $\dim X \le 3$.
\end{corollary}


\end{document}